\documentclass[11pt]{amsart}
\usepackage{graphicx}
\usepackage{amscd}
\usepackage{amsmath}
\usepackage{amsxtra}
\usepackage{amsfonts}
\usepackage{amssymb}
\usepackage{xcolor}
\usepackage{mathrsfs}  
\usepackage{leftindex}
\usepackage[euler]{textgreek}
\usepackage{placeins}

\usepackage{tikz-cd}

\newtheorem{theorem}{Theorem}[section]

\newtheorem{lemma}[theorem]{Lemma}
\newtheorem{proposition}[theorem]{Proposition}

\newtheorem{conjecture}[theorem]{Conjecture}
\theoremstyle{definition}
\newtheorem{definition}[theorem]{Definition}

\newtheorem{remark}[theorem]{Remark}

\newtheorem{example}[theorem]{Example}
\theoremstyle{remark}

\renewcommand{\theclaim}{\textup{\theclaim}}

\numberwithin{equation}{section}

\def\openone

{\mathchoice
	
	{\hbox{\upshape \small1\kern-3.3pt\normalsize1}}
	
	{\hbox{\upshape \small1\kern-3.3pt\normalsize1}}
	
	{\hbox{\upshape \tiny1\kern-2.3pt\SMALL1}}
	
	{\hbox{\upshape \Tiny1\kern-2pt\tiny1}}}

\makeatletter

\newbox\ipbox

\newcommand{\ip}[2]{\left\langle #1\, , \,#2\right\rangle}
\newcommand{\diracb}[1]{\left\langle #1\mathrel{\mathchoice
		
		{\setbox\ipbox=\hbox{$\displaystyle \left\langle\mathstrut
				#1\right.$}
			
			\vrule height\ht\ipbox width0.25pt depth\dp\ipbox}
		
		{\setbox\ipbox=\hbox{$\textstyle \left\langle\mathstrut
				#1\right.$}
			
			\vrule height\ht\ipbox width0.25pt depth\dp\ipbox}
		
		{\setbox\ipbox=\hbox{$\scriptstyle \left\langle\mathstrut
				#1\right.$}
			
			\vrule height\ht\ipbox width0.25pt depth\dp\ipbox}
		
		{\setbox\ipbox=\hbox{$\scriptscriptstyle \left\langle\mathstrut
				#1\right.$}
			
			\vrule height\ht\ipbox width0.25pt depth\dp\ipbox}
		
	}\right. }

\newcommand{\dirack}[1]{\left. \mathrel{\mathchoice
		
		{\setbox\ipbox=\hbox{$\displaystyle \left.\mathstrut
				#1\right\rangle$}
			
			\vrule height\ht\ipbox width0.25pt depth\dp\ipbox}
		
		{\setbox\ipbox=\hbox{$\textstyle \left.\mathstrut
				#1\right\rangle$}
			
			\vrule height\ht\ipbox width0.25pt depth\dp\ipbox}
		
		{\setbox\ipbox=\hbox{$\scriptstyle \left.\mathstrut
				#1\right\rangle$}
			
			\vrule height\ht\ipbox width0.25pt depth\dp\ipbox}
		
		{\setbox\ipbox=\hbox{$\scriptscriptstyle \left.\mathstrut
				#1\right\rangle$}
			
			\vrule height\ht\ipbox width0.25pt depth\dp\ipbox}
		
	} #1\right\rangle}

\newcommand{\beq}{\begin{equation}}
	
	\newcommand{\eeq}{\end{equation}}

\newcommand{\Zd}{\mathbb{Z}^d}

\newcommand{\bz}{\mathbb{Z}}

\newcommand{\br}{\mathbb{R}}

\newcommand{\bn}{\mathbb{N}}

\def\blfootnote{\xdef\@thefnmark{}\@footnotetext}

\DeclareMathOperator{\sinc}{sinc}

\renewcommand{\mod}{\operatorname{mod}}

\input xy
\xyoption{all}
\usepackage{amssymb}

\newcommand{\supp}{\operatorname*{supp}}
\newcommand{\norm}[1]{\lvert \lvert#1\rvert \lvert }

\def\R{\mathbb{R}}

\def\-{^{-1}}

\def\ty{\emptyset}

\def\Z{\mathbb{Z}}

\usepackage[OT2,T1]{fontenc}
\DeclareSymbolFont{cyrletters}{OT2}{wncyr}{m}{n}
\DeclareMathSymbol{\Sha}{\mathalpha}{cyrletters}{"58}

\newcommand{\Leb}{\textup{Leb}}

\begin{document}

	\title[Plancherel Identities for unbounded subsets of $\mathbb R^d$]{Plancherel Identities for unbounded subsets of $\mathbb R^d$}

	\author{Piyali Chakraborty}
	\address{[Piyali Chakraborty] University of Central Florida\\
		Department of Mathematics\\
		4000 Central Florida Blvd.\\
		P.O. Box 161364\\
		Orlando, FL 32816-1364\\
		U.S.A.\\} \email{Piyali.Chakraborty@ucf.edu}

	\author{Dorin Ervin Dutkay}
	\address{[Dorin Ervin Dutkay] University of Central Florida\\
		Department of Mathematics\\
		4000 Central Florida Blvd.\\
		P.O. Box 161364\\
		Orlando, FL 32816-1364\\
		U.S.A.\\} \email{Dorin.Dutkay@ucf.edu}


	\begin{abstract}
		
	We investigate spectral pairs associated with unbounded subsets of $\mathbb R^d$ that are periodic with respect to full-rank lattices. We establish a characterization of such spectral pairs in terms of the spectrality of an associated bounded fundamental component. More precisely, we prove that the Fourier transform defines an isometric isomorphism onto an $L^2$ space with an explicitly determined absolutely continuous pair measure, supported on a union of translates of a dual-lattice fundamental domain, if and only if the bounded component admits an orthogonal basis of exponential functions with frequencies in the dual lattice. In particular, the Plancherel identity alone is sufficient to characterize these spectral pairs. As an application, when the bounded component tiles $\br^d$ by a lattice, we obtain an equivalence between certain translational tiling properties of unbounded periodic sets and the existence of the corresponding Plancherel identities. We also examine periodicity properties of sets that tile by finite translation sets, highlighting distinctions between the one-dimensional and higher-dimensional settings.
		
	\end{abstract}
	\maketitle
	\tableofcontents

	\newcommand{\ati}{(A^T)^{-1}}
	\section{Introduction}

Fuglede's conjecture originated in a question posed by Irving Segal concerning the partial differential operators
\[
\frac{1}{2\pi i}\frac{\partial}{\partial x_j}
\]
defined on an open set \(\Omega \subset \mathbb{R}^d\). The question asks: under what conditions do these operators admit commuting self-adjoint extensions on \(L^2(\Omega)\)? In \cite{Fug74}, Fuglede proved that for finite-measure domains $\Omega$ with some regularity restrictions, there are such extensions if and only if there are orthogonal bases of exponential functions $\{e_\lambda: \lambda\in\Lambda\}$ for $L^2(\Omega)$, where $$e_\lambda(x)=e^{2\pi i\lambda\cdot x}.$$ He also proved that even for some nice domains like the disk or a triangle, this condition is {\it not} satisfied. It is satisfied, of course, by the unit cube, with the classical Fourier series. 
	
	\begin{definition}
		\label{defsp}
		A measurable set $\Omega\subset\br^d$ of finite Lebesgue measure is called \textit{spectral} if there exists a set $\Lambda\subset\br^d$ such that the family of exponential functions  
		$$\{e_\lambda : \lambda\in\Lambda\}$$  
		forms an orthogonal basis for $L^2(\Omega)$, $e_\lambda(x)=e^{2\pi i \lambda\cdot x}$. In this case, $\Lambda$ is called a \textit{spectrum} for $\Omega$.
		
		We say that $\Omega$ \textit{tiles} $\br^d$ by translations if there exists a set $\mathcal T\subset\br^d$ such that the translations $\{\Omega+t : t\in\mathcal T\}$ form a partition of $\br^d$ up to measure zero. The set $\mathcal T$ is then called a \textit{tiling set} for $\Omega$, and we say that $\Omega$ \textit{tiles $\br^d$ with the set} $\mathcal T$. 
		
		If only the condition that the sets $\Omega+t$, $t\in\mathcal T$ are mutually disjoint (up to measure zero) is satisfied, (so the union $\cup_t(\Omega+t)$ is not necessarily the whole space $\br^d$), then we say that $\Omega$ {\it packs} with $\mathcal T$.  
	\end{definition}
	
	Fuglede then formulated his conjecture: 
	
	\begin{conjecture}{\bf[Fuglede's Conjecture]}
		\label{conFu}
		A measurable subset $\Omega\subset\br^d$ of finite Lebesgue measure is spectral \textit{if and only if} it tiles $\br^d$ by translations.
	\end{conjecture}

	In the same paper, Fuglede proved that his conjecture is true when the spectrum or the tiling set is a lattice in $\br^d$.

	\begin{theorem}\label{thf1}
		
		Let $\Omega$ be a measurable subset of $\br^d$ of finite measure and let $A$ be some invertible $d\times d$ matrix. Then $\Omega$ tiles $\br^d$ with the lattice $A\bz^d$ if and only if $\Omega$ has spectrum the dual lattice $(A^T)^{-1}\bz^d$.
	\end{theorem}

	The conjecture was later shown to be false in dimensions $d\ge3$ \cite{Tao04, FMM06}, but it is true for convex domains \cite{LM22}.
	
		In \cite{Ped87}, Steen Pedersen extended Fuglede’s result by removing both the regularity condition and the finite-measure assumption. Naturally, the definition of a spectral set had to be adapted for infinite-measure domains (since $e_\lambda$ is not square-integrable in that case).
	
	\begin{definition}\label{defp1}
		For a function $f\in L^1(\br^d)$, define its classical Fourier transform as  
		$$\hat f(t)=\int_{\br^d}f(x)e^{-2\pi it\cdot x}\,dx,\quad(t\in\br^d).$$
		
		Let $\Omega\subset\br^d$ be measurable and let $\mu$ be a positive Radon measure on $\br^d$. We say that $(\Omega,\mu)$ is a \textit{spectral pair} if:  
		(1) for each $f\in L^1(\Omega)\cap L^2(\Omega)$, the continuous function $t\mapsto \hat f(t)$ satisfies $\int|\hat f|^2\,d\mu<\infty$; and  
		(2) the mapping $f\mapsto \hat f$ from $L^1(\Omega)\cap L^2(\Omega)\subset L^2(\Omega)$ into $L^2(\mu)$ is isometric and has dense range.
		
		This map then extends by continuity to an isometric isomorphism  
		$$\mathscr F:L^2(\Omega)\to L^2(\mu).$$
		
	The isometry property yields a Plancherel identity obtained by taking the Fourier transform of functions in $L^2(\Omega)$, restricting it to the support of the measure $\mu$ and integrating the absolute value squared of this restriction of the Fourier transform against the measure $\mu$.
		
		The set $\Omega$ is called \textit{spectral} if there exists a measure $\mu$ such that $(\Omega,\mu)$ forms a spectral pair; $\mu$ is then called a \textit{pair measure} or a {\it dual measure} for $\Omega$. When $\Omega$ has finite measure, this definition coincides with the earlier one \cite[Corollary 1.11]{Ped87}, and the pair measure $\mu$ is the counting measure on the spectrum $\Lambda$.
	\end{definition}

	There is an extensive literature on spectral sets of finite Lebesgue measure and their connections with translational tilings; see, for example, \cite{Fug74,JP92,Ped96,LM22} and the references therein. In contrast, spectral sets of infinite Lebesgue measure have received considerably less attention. Although the general framework for such spectral pairs was developed by Pedersen \cite{Ped87}, explicit constructions of infinite-measure spectral pairs, beyond the classical Fourier transform on $\mathbb R^d$, appear to be comparatively scarce in the literature. This provides one of the motivations for the present work, in which we construct and characterize classes of infinite-measure spectral pairs with explicitly determined absolutely continuous pair measures and investigate their connections with translational tilings.

	Denote by $\Leb_\Omega$ the Lebesgue measure restricted to a subset $\Omega$ of $\br^d$. In \cite{CD26}, we presented a class of unbounded sets in $\br$ that are spectral:
	
	\begin{theorem}\label{thm} \cite[Theorem 1.8]{CD26}
		Let $\Omega$ be an open subset of $\br$ and let $a\in\bn$, $a\geq 2$. Then $\Omega$ tiles $\br$ by $\{0,1,\dots,a-1\}$ if and only if $\Omega$ is a spectral set with pair measure $\mu$ equal to a renormalized Lebesgue measure on $\left[-\frac1{2a},\frac{1}{2a}\right]+\bz$, $\mu=a\Leb_{\left[-\frac1{2a},\frac1{2a}\right]+\bz}$ .
	\end{theorem}
	
	The main goal of this paper is to extend the preceding result to $\mathbb R^d$ and, more importantly, to establish a characterization of spectral pairs associated with unbounded periodic sets in terms of the spectral properties of their bounded fundamental components. Our main result, Theorem \ref{th1}, shows that the existence of an orthogonal basis of exponentials for the bounded component, with frequencies in the dual lattice, is equivalent to the existence of a Plancherel identity for the corresponding unbounded periodic set, with an explicitly determined absolutely continuous pair measure. A notable feature of this characterization is that the isometry property of the Fourier transform alone suffices to guarantee its surjectivity, and hence the full spectral-pair property. Thus, rather than merely constructing examples of infinite-measure spectral pairs, the theorem provides necessary and sufficient conditions for their existence within this class and establishes a direct connection between the spectral theory of bounded sets and Fourier analysis on unbounded periodic domains.
	
	\begin{theorem}\label{th1}
		Let $\Omega_0$ be a bounded open set in $\br^d$ of finite measure, let $\Lambda_0$ be a discrete subset of $\br^d$, and let $A$ be some invertible $d\times d$ matrix. Assume in addition that 
		\begin{equation}
			\label{eq1.1}
			\mbox{The set $\Omega_0$ packs with the lattice $A\bz^d$,}
		\end{equation}
		and 
		\begin{equation}
			\label{eq1.2}
			\mbox{$\lambda_{0}\cdot Ak\in\Z$ for all $\lambda_{0}\in\Lambda_{0}$ and $k\in\Zd$, equivalently $\Lambda_0\subseteq (A^T)^{-1}\bz^d$.}
		\end{equation}

		The following statements are equivalent:
		\begin{enumerate}
			\item $\Omega_0$ has spectrum $\Lambda_0$.
			\item The set $$\Omega=\Omega_0+A\bz^d$$ is spectral with pair measure 	
			$$\mu:=\frac{|\det A|}{|\Omega_{0}|}\Leb_{\ati[-\frac{1}{2},\frac{1}{2}]^{d}+\Lambda_{0}}.$$
			\item For the set $\Omega=\Omega_0+A\bz^d$, the restriction of the Fourier transform $$L^2(\Omega)\ni f\to \hat f|_{\ati[-\frac{1}{2},\frac{1}{2}]^{d}+\Lambda_{0}}\in L^2(\mu)$$ is an isometry.
		\end{enumerate}

	\end{theorem}
	
	As a fairly easy corollary, we obtain a generalization of Theorem \ref{thm}:
	
	\begin{theorem}
		\label{th2}
		Let $\Omega_0$ be an open subset of $\br^d$, let $A$ be an invertible $d\times d$ integer matrix, and let $\mathcal R$ be a complete set of representatives $\mod A\bz^d$. Assume in addition that the sets $\Omega_0$ packs with the lattice $A\bz^d$, and  define the set 
		$$\Omega:=\Omega_0+A\bz^d.$$

		The following statements are equivalent:
		\begin{enumerate}
			\item $\Omega$ tiles $\br^d$ with $\mathcal R$.
			\item $\Omega_0$ tiles $\br^d$ with the lattice $\bz^d$ and hence $|\Omega_0|=1$.
			\item $\Omega$ is spectral with pair measure 
			$$\mu:=\frac{|\det A|}{|\Omega_0|}\Leb_{\ati[-\frac12,\frac12]^d+\bz^d}.$$
			\item The Fourier transform $f\to \hat f$, from $L^1(\Omega)\cap L^2(\Omega)$ into $L^2(\mu)$ (with $\mu$ as above) extends to an isometry from $L^2(\Omega)$ into $L^2(\mu)$.
		\end{enumerate}
	\end{theorem}

	Note that our sets $\Omega$ are invariant under translations by elements from the lattice $A\bz^d$; in other words, any element of the lattice $A\bz^d$ is a period for $\Omega$ (see Definition \ref{defp}). In section \ref{sec4}, we show that, in dimension 1, if a set $\Omega$ tiles with a finite set, then $\Omega$ has to be periodic (Theorem \ref{thper}). If $\Omega$ tiles $\br$ by $\{0,1,\dots, a-1\}$ for some integer $a\geq 2$, then $a$ is a period for $\Omega$ (see the Proof of Theorem \ref{thm}). Also, if the cardinality of the tiling set $\mathcal T$ is a prime number $p$, and the greatest common divisor of the elements of $\mathcal T$ is 1, then $\Omega$ has period $p$ (Theorem \ref{thinv}).

	The paper is structured as follows: in Section \ref{sec2}, we introduce some definitions, notations, and some lemmas needed for the proofs of our main results; see in particular Lemma \ref{lemco} and Lemma \ref{lemc}, which might be of independent interest. In Section \ref{sec3}, we present the proofs of our main results. In Section 4, we show that, in dimension 1, the subsets of $\br$ that tile with a finite set have periodicity properties. In Section \ref{sec5}, we present some examples, and we make connections to fundamental domains of lattices and self-affine tiles. 
	
	\section{Preliminaries}\label{sec2}

	We begin with some  definitions and notations.

	For a measurable set $\Omega_0$, $|\Omega_0|$ denotes its Lebesgue measure. For a finite set $F$, $|F|$ denotes its cardinality. 
	
	For a finite or countable set $F$ in $\br^d$, we denote by $\Sha_F$ the Dirac comb on $F$, i.e., the counting measure on $F$. 
	
	When $F$ is finite, we say that $F$ has spectrum $\Lambda$ if the measure $\Sha_F$ has spectrum $\Lambda$, see Definition \ref{defsm} below. 
	
	For a Lebesgue measurable set $E$ in $\br^d$, we denote by $\Leb_E$ the restriction of the Lebesgue measure to the set $E$.
	
	For two finite or countable sets $F_1,F_2$ in $\br^d$ we use the notation $F_1\oplus F_2$ to indicate that each point $\lambda$ in $F_1+F_2$ can be written {\it uniquely} as $\lambda=f_1+f_2$ with $f_1\in F_1$ and $f_2\in F_2$.
	
	If $\Omega$ is a measurable subset in $\br^d$ and $\mathcal T$ is discrete, we write $\Omega\oplus\mathcal T$ to indicate that the sets $\Omega+t$, $t\in\mathcal T$ are disjoint, up to measure zero. In other words, for a.e., $y\in\Omega+\mathcal T$ there exist {\it unique} $x\in\Omega$ and $t\in\mathcal T$ such that $y=x+t$.
	
	For a full-rank lattice $\Gamma$ in $\br^d$, a {\it fundamental domain} is a measurable subset $F$ of $\br^d$ that tiles $\br^d$ with the lattice $\Gamma$.

	\begin{definition}
		\label{deft2}
		Let $\Omega_0$ and $\Omega$ be measurable subsets of $\br^d$ and $\mathcal T$ a discrete subset of $\br^d$. We say that $\Omega_0$ {\it tiles }$\Omega$ by $\mathcal T$ if $\Omega_0\oplus\mathcal  T=\Omega$ (up to Lebesgue measure zero). 
		
		Let $B_0$ and $B$  be subsets of a countable Abelian group $G$ and $\mathcal T\subseteq G$. We say that $B_0$ {\it tiles} $B$ by $\mathcal T$ if $B_0\oplus\mathcal T=B$.
	\end{definition}

	\begin{definition}
		\label{defp}
		A subset $B$ in a countable Abelian group $G$ is called {\it periodic} if there exists an element $g\in G$, $g\neq e$ such that $g+B=B$; in this case $g$ is called a {\it period} for $B$. The periods of $B$ together with identity element $e$ form a subgroup of $G$ called the {\it subgroup of periods of }$B$. 
		
		A measurable subset $\Omega$ of $\br^d$ is called {\it periodic} if there exists an element $p\in\br^d$, $p\neq 0$, such that $\Omega+p=\Omega$, up to measure zero. If $A$ is a $d\times d$ invertible real matrix,  we say that $\Omega$ has period lattice $A\bz^d$ if $Ak$ is a period for $\Omega$ for all $k\in\bz^d\setminus \{0\}$. 
		
		For a measurable subset $\Omega$ of $\br^d$ and some countable subset $\Gamma$ of $\br^d$, we say that $\Omega$ is {\it $\Gamma$-invariant} if $\Omega+\gamma\subseteq \Omega$ for all $\gamma\in\Gamma$, up to measure zero.

			For an invertible $d\times d$ integer matrix $A$, a {\it complete set of representatives (or residues) }$\mod A\bz^d$, is a set $\mathcal R$ contained in $\bz^d$ such that for any $k\in\bz^d$, there exists a unique $r\in\mathcal R$ such that $k-r\in A\bz^d$.

		Let $A$ be an invertible $d\times d$ real matrix. We say that two measurable subsets $E$ and $F$ of $\br^d$ are {\it congruent} $\mod A\bz^d$, if there exist partitions $\{E_k: k\in\bz^d\}$ of $E$, and $\{F_k : k\in \bz^d\}$ of $F$ such that 
		$$F_k=E_k+Ak,\quad( k\in\bz^d),$$
		up to measure zero.

	\end{definition}

	For a measurable set $\Omega$ in $\br^d$ and a positive Radon measure $\mu$ on $\br^d$, we denote by $\mathscr F_{\Omega,\mu}$ the restriction the map from $L^1(\Omega)\cap L^2(\Omega)$ into $L^2(\mu)$, defined by the Fourier transform $$\mathscr F_{\Omega,\mu}f=\hat f,$$ as long as it is well-defined.

		\begin{definition}\label{defsm}
		Let $\mu$ be a finite positive Radon measure on $\br^d$. We say that $\mu$ is a {\it spectral measure}, if there exists a subset $\Lambda$ in $\br^d$ such that the family of exponential functions 
		$$\{e_\lambda : \lambda\in\Lambda\}$$
		forms an orthogonal basis for the Hilbert space $L^2(\mu)$. In this case, $\Lambda$ is called a {\it spectrum} for the measure $\mu$.
	\end{definition}
	
 The next lemma is well known, see e.g. \cite{JP98a}.
		\begin{lemma}
		If $\mu $ is a finite positive Radon measure on $\mathbb{R}^{d}$ with total measure $\mu(\mathbb{R}^{d})=m$, then $\mu$ has spectrum $\Lambda$ if and only if $$\sum_{\lambda\in \Lambda}|\hat{\mu}(t+\lambda)|^{2}=m^{2}.$$ 
		\label{lemma1}
	\end{lemma}
	\begin{proof}
		Suppose $\Omega$ has a spectrum $\Lambda.$ Then 
		$$\{\frac{1}{\sqrt{m}}e_\lambda : \lambda\in\Lambda\}$$ is an orthonormal basis for $L^2(\mu)$. Applying the Parseval identity to the function $e_{-t}$, $t\in\br^{d}$ we obtain 
		$$m=\|e_{-t}\|^2=\sum_{\lambda\in\Lambda}\left|\ip{e_{-t}}{\frac{1}{\sqrt{m}}e_\lambda}_{L^2(\mu)}\right|^2$$ $$=\frac{1}{m}\sum_{\lambda\in\Lambda}|\int_{\R^{d}}e^{-2\pi i (t+\lambda)\cdot u}d\mu(u)|^{2}=\frac{1}{m}\sum_{\lambda\in\Lambda}\left|\hat\mu(t+\lambda)\right|^2$$ and therefore $$ \sum_{\lambda\in\Lambda}\left|\hat\mu(t+\lambda)\right|^2=m^{2}.$$
		Conversely, let $$\sum_{\lambda\in \Lambda}|\hat{\mu}(t+\lambda)|^{2}=m^{2}.$$
		Substitute $t=-\lambda_{0}$ and since $\hat\mu(0)=\mu(\br^d)=m$, we get that $\hat\mu(-\lambda_0+\lambda)=0$ for $\lambda\neq \lambda_0$, and this proves that the exponentials $e_\lambda$ and $e_{\lambda_0}$ are orthogonal. Also, $\norm{e_{-t}}^{2}=m$ and $$\sum_{\lambda\in\Lambda}\left|\ip{e_{-t}}{\frac{1}{\sqrt{m}}e_\lambda}_{L^2(\mu)}\right|^2=\frac{1}{m}\sum_{\lambda\in\Lambda}|\int_{\R^{d}}e^{-2\pi i (t+\lambda)\cdot u}d\mu(u)|^{2}=\frac{1}{m}\sum_{\lambda\in\Lambda}\left|\hat\mu(t+\lambda)\right|^2=m,$$ which shows that $$\|e_{-t}\|^2=\sum_{\lambda\in\Lambda}\left|\ip{e_{-t}}{\frac{1}{\sqrt{m}}e_\lambda}_{L^2(\mu)}\right|^2.$$
		On the right hand side we have the norm squared of the projection of $e_{-t}$ onto the closed span of the functions $\{e_\lambda :\lambda\in\Lambda\}$. This shows that $e_{-t}$ is in this closed span. By the Stone-Weierstrass theorem and by approximation by continuous compactly supported functions, the functions $e_{-t}$ span the entire space $L^2(\mu)$. It follows that $\{e_\lambda :\lambda\in\Lambda\}$ also span the entire space $L^2(\mu)$ and therefore they form an orthogonal basis for it.

	\end{proof}

	The next lemma shows that we can convolute a spectral measure with a spectral finite set to obtain a new spectral measure whose spectrum is the direct sum of the two spectra. 
	\begin{lemma}\label{lemco}
		Let $F$ be a finite subset of $\br^d$, and suppose the measure $\Sha_F$ has spectrum $\Lambda_1$. Let $\mu_2$ be a finite positive Radon measure on $\br^d$ with spectrum $\Lambda_2$. Assume in addition that 
		\begin{equation}\label{eqco1}
			f\cdot\lambda_2\in\bz,\mbox{ for all }f\in F\mbox{ and }\lambda_2\in\Lambda_2.
		\end{equation}
		Then the convolution measure $\Sha_F\ast \mu_2$ has spectrum $\Lambda_1\oplus\Lambda_2$. 
	\end{lemma}

	\begin{proof}
		
		Note first that, for $t\in\br^d$ and $\lambda_2\in\Lambda_2$, with \eqref{eqco1},
		$$\hat\Sha_F(t+\lambda_2)=\sum_{f\in F}e^{-2\pi i f\cdot (t+\lambda_2)}=\sum_{f\in F}e^{-2\pi i f\cdot t}=\hat\Sha_F(t).$$
		Using Lemma \ref{lemma1} we compute:
		$$\sum_{\lambda_1\in\Lambda_1}\sum_{\lambda_2\in\Lambda_2}\left|\widehat{\Sha_F\ast\mu_2}(t+\lambda_1+\lambda_2)\right|^2$$
		$$=\sum_{\lambda_1\in\Lambda_1}\sum_{\lambda_2\in\Lambda_2}\hat\Sha_F(t+\lambda_1+\lambda_2)|^2|\hat\mu_2(t+\lambda_1+\lambda_2)|^2$$
		$$=\sum_{\lambda_1\in\Lambda_1}|\hat\Sha_F(t+\lambda_1)|^2\sum_{\lambda_2\in\Lambda_2}|\hat\mu_2(t+\lambda_1+\lambda_2)|^2$$
		$$=\sum_{\lambda_1\in\Lambda_1}|\hat\Sha_F(t+\lambda_1)|^2\mu_2(\br^d)^2=|A|^2\mu_2(\br^d)^2=(\Sha_F\ast\mu_2(\br^d))^2.$$
		
		This proves that $\Lambda_1+\Lambda_2$ is a spectrum for the convolution. 
		
		To show that the sum $\Lambda_1\oplus\Lambda_2$ is direct, assume now that $\lambda_1+\lambda_2=\lambda_1'+\lambda_2'$ for some $\lambda_1,\lambda_1'\in\Lambda_1$ and $\lambda_2,\lambda_2'\in\Lambda_2$, $(\lambda_1,\lambda_2)\neq(\lambda_1',\lambda_2')$. Take $t=-\lambda_1-\lambda_2$ in the previous sequence of equalities. Then, the sum on the left of the equalities above, contains the terms $$|\widehat{\Sha_F\ast\mu_2}(t+\lambda_1+\lambda_2)|^2+|\widehat{\Sha_F\ast\mu_2}(t+\lambda_1'+\lambda_2')|^2=2|\widehat{\Sha_F\ast\mu_2}(0)|^2=2|\Sha_F\ast\mu_2(\br^d)|^2,$$
		and this yields a contradiction.

	\end{proof}

	The next two lemmas show  that translation of a pair measure is another pair measure for the same set, and if we dilate a spectral set, then the dilated set has spectrum the contraction of the pair measure with the inverse transpose of the dilation matrix. 
	
	\begin{lemma}
		\label{lemt}
		Let $\mu$ be a positive Radon measure on $\br^d$. For $t\in\br^d$, define the translation measure 
		$$T_t\mu(E)=\mu(E+t),\mbox{ for Borel subsets $E$ of $\br^d$}.$$
	Let $\Omega$ a measurable subset of $\br^d$. If the Fourier transform $\mathscr F_{\Omega,\mu}$ extends to an isometric/onto map between $L^2(\Omega)$ and $L^2(\mu)$, then so does $\mathscr F_{\Omega,T_t\mu}$. In particular, if $\mu$ is a pair measure for $\Omega$ then $T_t\mu$ is also a pair measure for $\Omega$. 
	\end{lemma}
	
	\begin{proof}
Let $t\in\br^d$.		Note that, for a Borel set $E$, 
		$$\int \chi_E\,dT_t\mu=T_t\mu(E)=\mu(E+t)=\int \chi_{E+t}\,d\mu=\int \chi_E(x-t)\,d\mu(x),$$
		and so 
		$$\int f\,d T_t\mu=\int f(x-t)\,d\mu(x),\quad (f\in L^1(T_t\mu)).$$
		
		Define the operator $\tilde T_t$ from $L^2(\mu)$ to $L^2(T_t\mu)$ by 
		$$\tilde T_th(x)=h(x+t),\quad (x\in\br^d, h\in L^2(\mu)).$$
		This is an isometric isomorphism. 
		
		Define the modulation operator $\mathscr M_t$ on $L^2(\Omega)$:
		$$\mathscr M_t f=e_{-t}f,\quad (f\in L^2(\Omega)).$$
		This is also an isometric isomophism.

		A simple computation shows that 
		$$\mathscr F_{\Omega,T_t\mu}\circ \mathscr M_t=\tilde T_t\circ \mathscr F_{\Omega,\mu},$$
		i.e., the following diagram is commutative:
		
		\begin{center}
			
			\begin{tikzcd}
				L^2(\Omega) \arrow[r, "\mathscr M_t"] \arrow[d, "\mathscr F_{\Omega,\mu}"'] 
				& L^2(\Omega) \arrow[d, "\mathscr F_{\Omega,T_t\mu}"] \\
				L^2(\mu) \arrow[r, "\tilde T_t"'] 
				& L^2(T_t\mu)
			\end{tikzcd}
		\end{center}
		It follows that $\mathscr F_{\Omega,T_t\mu}$ is isometric/onto if $\mathscr F_{\Omega,\mu}$ is. 
	
	\end{proof}

	\begin{lemma}
		\label{lemdi}
		Let $\mu$ be a positive Radon measure on $\br^d$. Let $M$ be an invertible $d\times d$ real matrix. Define the dilation measure $D_M\mu$ by 
		$$\int g\,d D_M\mu=\frac{1}{|\det M|}\int g((M^T)^{-1}x)d\mu(x), \mbox{ whenever the function }x\to g((M^T)^{-1}x)\mbox{ is in } L^1(\mu).$$
		Let $\Omega$ be a measurable subset of $\br^d$. If the Fourier transform $\mathscr F_{\Omega,\mu}$ extends to an isometric/onto map between $L^2(\Omega)$ and $L^2(\mu)$, then so does $\mathscr F_{M\Omega,D_M\mu}$. If $\Omega$ is spectral with pair measure $\mu$, then $M\Omega$ is spectral with pair measure $D_M\mu$.  
	\end{lemma}
	
	\begin{proof}
		Define the operator $D_M: L^2(\Omega)\rightarrow L^2(M\Omega)$ by 
		$$D_Mf (x)=\frac{1}{\sqrt{|\det M|}}f(M^{-1}x),\quad (x\in M\Omega, f\in L^2(\Omega)).$$
			$D_M$ is an isometric isomorphism, as can be seen by a change of variable.

		We also denote by $\tilde D_M$ the operator from $L^2(\mu)$ to $L^2(D_M\mu)$, defined by 
		$$\tilde D_M h(x)=\sqrt{|\det M|}h(M^Tx),\quad (x\in \br^d,h\in L^2(\mu)).$$
		This is also an isometric isomorphism.

			A quick computation shows that 
			$$\mathscr F_{M\Omega,D_M\mu}\circ D_M=\tilde D_M\circ\mathscr F_{\Omega,\mu},$$
			i.e., the diagram is commutative:
		\begin{center}
			
			\begin{tikzcd}
				L^2(\Omega) \arrow[r, "D_M"] \arrow[d, "\mathscr F_{\Omega,\mu}"'] 
				& L^2(M\Omega) \arrow[d, "\mathscr F_{M\Omega,D_M\mu}"] \\
				L^2(\mu) \arrow[r, "\tilde D_M"'] 
				& L^2(D_M\mu)
			\end{tikzcd}
		\end{center}
			It follows that $\mathscr F_{M\Omega,D_M\mu}$ is isometric/onto if $\mathscr F_{\Omega,\mu}$ is. 
	\end{proof}

	The next lemma shows that if some renormalized Lebesgue measures are in duality (actually just the isometry property is enough), and the pair measure has a tiling property, then the normalization constant $c$ is completely determined by the tiling property, and it has to be equal to the cardinality of the tiling set $|F|$. 
	
	\begin{lemma}
		\label{lemc}
		Let $\Omega$ and $\tilde \Omega$ be measurable subsets of $\br^d$ and $c$ some positive constant. Let $\mu$ be $c\cdot\Leb|_{\tilde\Omega}$. Assume that the Fourier transform $\mathscr F_{\Omega,\mu}$ extends to an isometry from $L^2(\Omega)$ to $L^2(\mu)$. Assume in addition that $\tilde\Omega$ tiles $\br^d$ with a finite set $F$. Then 
		$$c=|F|.$$
	\end{lemma}
	
	\begin{proof}
		By Lemma \ref{lemt}, we have that $\mathscr F_{\Omega,T_{-f}\mu}$ also extends to an isometry, for any $f\in F$. Note that $T_{-f}\mu=c\Leb|_{\tilde\Omega+f}$.

		Take now some non-zero function $\varphi\in L^2(\Omega)$. We have 
		$$\|\varphi\|_{L^2(\Omega)}^2=\int_{\br^d}|\varphi|^2=\int_{\br^d}|\hat \varphi|^2=\sum_{f\in F}\int_{\tilde\Omega+f}|\hat \varphi|^2=\sum_{f\in F}\frac{1}{c}\int |\hat \varphi|^2\,d T_{-f}\mu$$$$=\frac1c\sum_{f\in F}\|\varphi\|_{L^2(\Omega)}^2=\frac{|F|}{c}\|\varphi\|_{L^2(\Omega)}^2.$$
		Thus $c=|F|$.
	\end{proof}

	The next lemma is also well known, see e.g. \cite[page 354]{LaWa96}.
	\begin{lemma}\label{lemcg}
	Let $A$ be a $d\times d$ invertible real matrix	and $\Omega_0$ a measurable subset of $\br^d$. Suppose $\Omega_0$ tiles $\br^d$ by $A\bz^d$. Then $\Omega_0$ is congruent to $A(0,1)^d$ $\mod A\bz^d$. In particular, 
	$$|\Omega_0|=|\det A|.$$
	\end{lemma}
	
	\begin{proof}
		By using a substitution $u=Ax$, we can assume that $A$ is the identity and $\Omega_0$ tiles by $\bz^d$. 
	
Define the sets, for $k\in\bz^d$, 
$$E_k:=\{x\in \Omega_0 : x+k\in(0,1)^d\},\quad F_k=E_k+k.$$
We claim that $\{E_k : k\in\bz^d\}$ is a partition of $\Omega_0$ and $\{F_k : k\in\bz^d\}$ is a partition of $(0,1)^d$. 

Indeed, the fact that $(0,1)^d$ tiles with $\bz^d$ implies that $\{E_k\}_k$ is a partition of $\Omega_0$: if $x\in E_k\cap E_{k'}$, then $x+k, x+k'\in (0,1)^d$, which implies that $k=k'$. Thus the sets $E_k$ are mutually disjoint.  Also, if $x\in \Omega_0$ then there exists $k\in\bz^d$ such that $x+k\in(0,1)^d$, so $x\in E_k$. Thus $\cup_k E_k=\Omega_0$.

Similarly, the fact that $\Omega_0$ tiles with $\bz^d$ implies that $\{F_k\}_k$ is a partition of $(0,1)^d$. This means that $\Omega_0$ and $(0,1)^d$ are congruent $\mod \bz^d$.

Finally 
$$|\Omega_0|=\sum_{k\in\bz^d}|E_k|=\sum_{k\in\bz^d}|F_k|=|(0,1)^d|=1.$$
	\end{proof}

	\section{Proofs}\label{sec3}

	\begin{proof}[Proof of Theorem \ref{th1}]
		
		(i)$\Rightarrow$(ii). 	
First we define a nested sequence of spectral subsets of $\Omega$ which cover $\Omega$.		
	
	For $n\in\bn$, let 
	$$\Omega_{n}=\Omega_{0}+AL_n,\mbox{ where }L_n=\{-n,\dots,n\}^d.$$
	Clearly, the sets $\Omega_{n}$ are increasing and they cover $\Omega$. Also, with Lemma \ref{lemco}, since the set $AL_n$ has spectrum $\frac1{2n+1}\ati L_n$, we get that the set $\Omega_{n}$ has spectrum $$\Lambda_{n}=\frac{1}{2n+1}(\ati\{-n,\dots,n\}^{d})+\Lambda_{0}.$$
	\medskip
	
{\bf Isometry.}		Let $f\in C_{c}^{\infty}(\Omega)$. Let $n$ be big enough so that $\supp f\subseteq\Omega_{n}=\Omega_{0}+A\{-n,\dots,n\}^{d}.$ Since $\Omega_{n}$ has spectrum $\frac{1}{2n+1}(A^{T})^{-1}L_{n}+\Lambda_{0},$ we have $$\int_{\Omega}|f(x)|^{2}dx=\int_{\Omega_{n}}|f(x)|^{2}dx=\sum_{l\in L_{n},k\in\Lambda_{0}}|\ip{f}{\frac{1}{\sqrt{(2n+1)^{d}|\Omega_{0}|}}e_{\frac{1}{2n+1}(A^{T})^{-1}l+k}}|^{2}$$ $$=\sum_{l\in L_{n},k\in\Lambda_{0}}\frac{|\det A|}{|\Omega_{0}|}|\hat{f}(\frac{1}{2n+1}(A^{T})^{-1}l+k)|^{2}\frac{1}{(2n+1)^{d}|\det A|}.$$
	This is the Riemann sum on the set $(A^{T})^{-1}[-\frac{1}{2},\frac{1}{2}]^{d}+\Lambda_{0}$ with intermediate points $\frac{1}{2n+1}(A^{T})^{-1}l+k,l\in L_{n},k\in\Lambda_{0}$ on the partition $$\left\{\frac{1}{2n+1}(A^{T})^{-1}([l_{1}-\frac{1}{2},l_{1}+\frac{1}{2}]\times[l_{2}-\frac{1}{2},l_{2}+\frac{1}{2}]\times\cdots\times[l_{d}-\frac{1}{2},l_{d}+\frac{1}{2}])+k:(l_{1},\dots,l_{d})\in L_{n},k\in\Lambda_{0}\right\}.$$
	Note also that in each coordinate the leftmost endpoint is $\frac{-n-\frac{1}{2}}{2n+1}=-\frac{1}{2},$ and the rightmost endpoint is $\frac{n+\frac{1}{2}}{2n+1}=\frac{1}{2}.$ Also note that the sets $(A^{T})^{-1}[-\frac{1}{2},\frac{1}{2}]^{d}+k,k\in\Lambda_{0}$ are disjoint because $k$ is in the dual lattice $(A^{T})^{-1}\Zd$, as $\Lambda_{0}\subseteq (A^{T})^{-1}\Zd.$ Each interval in this partition has volume $\frac{1}{|\det A|}\frac{1}{(2n+1)^{d}}.$ This shows that we do indeed have Riemann sums and they converge to $$\frac{|\det A|}{|\Omega_{0}|}\int_{(A^{T})^{-1}[-\frac{1}{2},\frac{1}{2}]^{d}+\Lambda_{0}}|\hat{f}(t)|^{2}dt.$$ As $C^{\infty}_{c}(\Omega)$ is dense in $L^{2}(\Omega),$ the Fourier transform between the two $L^{2}$ spaces is isometric.
	
	\medskip

{\bf Onto.}	Note that we have for $k,k_{0}\in\Lambda_{0}$
	\begin{equation}
		\label{5.1}
		\hat{\chi}_{\Omega_{0}}(k-k_{0})=\int_{\Omega_{0}}e^{-2\pi i(k-k_{0})\cdot x}dx=\ip{e_{k_{0}}}{e_{k}}_{L^{2}(\Omega^0)}=\delta_{k,k_{0}}|\Omega_{0}|.
	\end{equation}
	The characteristic function $$\chi_{\Omega}=\chi_{\Omega_{0}+A\mathbb{Z}^{d}}$$ is periodic with period lattice $A\mathbb{Z}^{d}$, therefore its Fourier transform as a tempered distribution is :  $$\hat{\chi}_{\Omega}=\frac{1}{|\det A|}\sum_{k\in \mathbb{Z}^{d}}\hat{\chi}_{\Omega_{0}}((A^{T})^{-1}k)\delta_{(A^{T})^{-1}k}.$$
	This follows from the next:
	\begin{lemma}
		Let $\Tilde{h}$ be a function on $\mathbb{R}^{d}$ with period lattice $A\mathbb{Z}^{d}$ and let $h$ be its restriction to the period interval $h=\Tilde{h}|_{A(0,1)^{d}}\in L^{2}(A(0,1)^{d}).$ Represent $\Tilde{h}$ in the Fourier basis $\frac{1}{\sqrt{|\det A|}}e_{(A^{T})^{-1}k}$, $k\in\mathbb{Z}^{d}$:
		$$\Tilde{h}=\sum_{k\in \mathbb{Z}^{d}}\frac{1}{|\det A|}\langle h,e_{(A^{T})^{-1}k}\rangle_{L^{2}(A(0,1)^{d})}e_{(A^{T})^{-1}k}=\frac{1}{|\det A|}\sum_{k\in \mathbb{Z}^{d}}\hat{h}((A^{T})^{-1}k)e_{(A^{T})^{-1}k}.$$
		Then the Fourier transform of $\Tilde{h}$ as a tempered distribution is $$\hat{\Tilde{h}}=\frac{1}{|\det A|}\sum_{k\in\mathbb{Z}^{d}}\hat{h}((A^{T})^{-1}k)\delta_{(A^{T})^{-1}k}.$$
	\end{lemma}
	\begin{proof}
		Let $\Lambda_{\tilde h}$ be tempered distribution associated to $\tilde h$, $$\Lambda_{\tilde h}(\psi)=\int_{\R^{d}}\tilde h(x)\psi(x)dx.$$ Then, for a Schwartz function $\psi$ $$\hat{\Lambda}_{\tilde h}(\psi)=\Lambda_{\tilde h}(\hat{\psi})=\frac{1}{|\det A|}\sum_{k\in \Zd}\hat{h}((A^{T})^{-1}k)\int_{\R^{d}}e^{2\pi i(A^{T})^{-1}k\cdot t }\hat{\psi}(t)dt$$ 
		$$=\frac{1}{|\det A|}\sum_{k\in \Zd}\hat{h}((A^{T})^{-1}k)\psi((A^{T})^{-1}k).$$ The sums are convergent as $\sum_{k\in\Zd}|\hat{h}((A^{T})^{-1}k)|^2<\infty$ and $\psi$ is a Schwartz function. 
	\end{proof}
	Take now some $C_{c}^{\infty}$ function $g$ supported on the interior of the subsets $(A^{T})^{-1}[-\frac{1}{2},\frac{1}{2}]^{d}+k_{0}$ for some $k_{0}\in \Lambda_{0}$, and let $\psi$ be its inverse Fourier transform; it is a Schwartz function. Let $f=|\det A|\chi_{\Omega}\cdot\psi$ be the restriction to $\Omega$. We will show that $\hat{f}|_{\ati[-\frac{1}{2},\frac{1}{2}]^{d}+\Lambda_{0}}=|\Omega_{0}|\cdot g.$ 
	
	We have $$\hat{f}(t)=|\det A|\hat{\chi}_{\Omega}*\hat{\psi}(t)=|\det A|(\frac{1}{|\det A|}\sum_{k\in \Zd}\hat{\chi}_{\Omega_{0}}((A^{T})^{-1}k)\cdot \delta_{(A^{T})^{-1}k})*g(t)$$ $$=\sum_{k\in \Zd}\hat{\chi}_{\Omega_{0}}((A^{T})^{-1}k)g(t-(A^{T})^{-1}k).$$ We need to restrict this to $(A^{T})^{-1}[-\frac{1}{2},\frac{1}{2}]^{d}+\Lambda_{0}.$ 
	
	Note that $g(t-(A^{T})^{-1}k)$ is supported on $(A^{T})^{-1}[-\frac{1}{2},\frac{1}{2}]^{d}+(A^{T})^{-1}k+k_{0}.$ These sets are disjoint for different $k$. When we restrict to $(A^{T})^{-1}[-\frac{1}{2},\frac{1}{2}]^{d}+\Lambda_{0}$, we keep only the terms with $(A^{T})^{-1}k+k_{0}\in\Lambda_{0}$, so $k=A^{T}(l-k_{0})$ for some $l\in\Lambda_{0}$.
	
	So,
	$$\hat{f}\cdot\chi_{(A^{T})^{-1}[-\frac{1}{2},\frac{1}{2}]^{d}+\Lambda_{0}}(t)=\sum_{l\in \Lambda_{0}}\hat{\chi}_{\Omega_{0}}((A^{T})^{-1}A^{T}(l-k_{0}))\cdot g(t-(A^{T})^{-1}A^{T}(l-k_{0}))=\sum_{l\in \Lambda_{0}}\hat{\chi}_{\Omega_{0}}(l-k_{0})\cdot g(t-(l-k_{0})).$$
	
	Note that by \eqref{5.1}, $\hat{\chi}_{\Omega_{0}}(l-k_{0})=0$ unless $l=k_{0}$, and $\hat\chi_{\Omega_0}(0)=|\Omega_0|$, and thus $$\hat{f}|_{(A^{T})^{-1}[-\frac{1}{2},\frac{1}{2}]^{d}+\Lambda_0}=g(t)\cdot|\Omega_{0}|.$$ It follows that the function $g$ is in the range. 
	
	Since the linear combination of functions  like $g$ are dense, we get that the map is onto.

	(ii)$\Rightarrow$(iii) is trivial.

	(iii)$\Rightarrow$(i). 	
		Since the sets $\Omega_0+Ak$, $k\in\bz^d$ are mutually disjoint, we have, for $k\in\bz^d$, and $l\in\bz^d$, 
	$$|\Omega_0|\delta_k=\ip{e_{(A^T)^{-1}l}\chi_{\Omega_0}}{e_{(A^T)^{-1}l}T_{Ak}\chi_{\Omega_0}}$$
	and using the Fourier transform isometry,
	$$=\int_{(A^T)^{-1}\left[-\frac12,\frac12\right]^d+\Lambda_0}e^{2\pi i Ak\cdot (t-(A^T)^{-1}l)}|\hat\chi_{\Omega_0}(t-(A^T)^{-1}l)|^2\frac{|\det A|}{|\Omega_0|}dt$$
	and since $l\in\bz^d$, 
	$$=\frac{|\det A|}{|\Omega_0|}\sum_{\lambda\in\Lambda_0}\int_{(A^T)^{-1}\left[-\frac12,\frac12\right]^d}e^{2\pi i Ak\cdot (t+\lambda_0)}|\hat\chi_{\Omega_0}(t-(A^T)^{-1}l+\lambda_0)|^2\,dt$$
	and because $\Lambda_0\subseteq (A^T)^{-1}\bz^d$, 
	$$=\frac{|\det A|}{|\Omega_0|}\int_{(A^T)^{-1}\left[-\frac12,\frac12\right]^d}e^{2\pi i Ak\cdot t}\sum_{\lambda\in\Lambda_0}|\hat\chi_{\Omega_0}(t-(A^T)^{-1}l+\lambda_0)|^2\,dt$$
	and with the substitution $A^Tt=u$,
	$$=\frac{1}{|\Omega_0|}\int_{\left[-\frac12,\frac12\right]^d}e^{2\pi i k\cdot u}\sum_{\lambda\in\Lambda_0}|\hat\chi_{\Omega_0}((A^T)^{-1}u-(A^T)^{-1}l+\lambda_0)|^2\,du.$$
	
	Thus, the Fourier coefficients of the function $\sum_{\lambda_0\in\Lambda_0}|\hat\chi_{\Omega_0}((A^T)^{-1}u-(A^T)^{-1}l+\lambda_0)|^2$, $u\in\left[-\frac12,\frac12\right]^d$, are $|\Omega_0|^2\delta_k$, $k\in\bz^d$. 
	
	Therefore 
	$$\sum_{\lambda_0\in\Lambda_0}|\hat\chi_{\Omega_0}((A^T)^{-1}u-(A^T)^{-1}l+\lambda_0)|^2=|\Omega_0|^2,\quad (u\in\left[-\frac12,\frac12\right]^d,l\in\bz^d),$$
	so 
	$$\sum_{\lambda_0\in\Lambda_0}|\hat\chi_{\Omega_0}(t+\lambda_0)|^2=|\Omega_0|^2,\quad(t\in\br^d),$$
	and with Lemma \ref{lemma1}, we get that $\Lambda_0$ is a spectrum for $\Omega_0$.

	\end{proof}

		\begin{remark}
		By Lemma \ref{lemt}, if $\Omega$ has pair measure $c\cdot \Leb_{\ati[-1/2,1/2]^d+\Lambda_0}$, then any translation of this measure is also a pair measure for $\Omega$, and therefore we can replace $[-1/2,1/2]^d$ by $[0,1]^d$ or any cube with side 1. 
	\end{remark}

	The analogous result to Theorem \ref{th1} for tiles is much easier to prove, we include it here. 
	
	\begin{proposition}\label{pr2}
	 Let $\Omega_0$ be an open subset of $\br^d$, 	let $A$ be an invertible $d\times d$ integer matrix and $B$ a subset of $\br^d$. 
	  Assume in addition that the sets $\Omega_0+Ak$, $k\in\bz^d$, are mutually disjoint, and define the set
	   
	  $$\Omega:=\Omega_0+A\bz^d.$$
	  
	  The following statements are equivalent 
	  \begin{enumerate}
	  	\item $\Omega$ tiles $\br^d$ by $B$. 
	  	\item $\Omega_0$ tiles $\br^d$ by $A\bz^d\oplus B$.
	  \end{enumerate}
	  In this case, the sets $A\bz^d+b$, $b\in B$, are mutually disjoint. 
	\end{proposition}

	\begin{proof}
		(i)$\Rightarrow$(ii). Let $k,k'\in\bz^d$ and $b,b'\in B$. Assume that $\Omega_0+Ak+b$ intersects $\Omega_0+Ak'+b'$ non-trivially. Then $\Omega+b$ intersects $\Omega+b'$ non-trivially, so $b=b'$. Then $\Omega+Ak$ intersects $\Omega+Ak'$ non-trivially and, by hypothesis, $k=k'$. In particular, the sets $A\bz^d+b$ and $A\bz^d+b'$ are disjoint. 
		
		Also, $\Omega_0+A\bz^d+B=\Omega+B=\br^d$. Thus, $\Omega_0$ tiles by $A\bz^d\oplus B$.

		(ii)$\Rightarrow$(i). Let $b,b'\in B$. Assume that $\Omega+b$ and $\Omega+b'$ intersect non-trivially. Since $\Omega=\Omega_0+A\bz^d$, this means that, for some $k,k'\in\bz^d$, $\Omega_0+Ak+b$ and $\Omega_0+Ak'+b'$ intersect non-trivially. Since $\Omega_0$ is a tile, we get that $k=k'$ and $b=b'$. 
		
		Also, $\Omega+B=\Omega_0+A\bz^d+B=\br^d$. Thus $\Omega$ tiles $\br^d$ by $B$.
	\end{proof}

	\begin{proof}[Proof of Theorem \ref{th2}]
		(i)$\Leftrightarrow$(ii) follows from Proposition \ref{pr2}: $\Omega$ tiles with $\mathcal R$ if and only if $\Omega_0$ tiles with $A\bz^d\oplus\mathcal R=\bz^d$. 
		Also, by Lemma \ref{lemcg}, if $\Omega_0$ tiles with $\bz^d$ then $|\Omega_0|=1$.

		(ii)$\Rightarrow$(iii). We have $\Omega=\Omega_0\oplus A\bz^d$. 
		If $\Omega_0$ tiles with the lattice $\bz^d$, then by Theorem \ref{thf1}, it follows that $\Omega_0$ has spectrum $\Lambda_0=\bz^d$, and $|\Omega_0|=1$. And then, with Theorem \ref{th1}(i)$\Rightarrow$(ii), we get (iii).
		
		(iii)$\Rightarrow$(iv) is trivial.

		(iv)$\Rightarrow$(ii). Follows from Theorem \ref{th1}(iii)$\Rightarrow$(i).
	\end{proof}

	\begin{remark}
		From the equivalences in Theorem \ref{th2}, we see that under the hypotheses in (iii) we get that $\Omega_0$ has measure 1. But here is a more direct way:

		Assume that  $\Omega=\Omega_0+A\bz^d$ has pair measure $\frac{|\det A|}{|\Omega_0|}\Leb_{\ati[-1/2,1/2]^d+\bz^d}$. We can also see that $\ati[-1/2,1/2]^d+\bz^d$ tiles $\br^d$ with $\ati \tilde {\mathcal R}$, where $\tilde {\mathcal R}$ is any complete set of representatives $\mod A^T\bz^d$. Then, with Lemma \ref{lemc}, we get that $\frac{|\det A|}{|\Omega_0|}=|\tilde{\mathcal R}|=|\det A|$. Thus $|\Omega_0|=1$. 
	\end{remark}

	With Theorem \ref{th2}, we obtain Theorem \ref{thm} as a corollary.

	\begin{proof}[Proof of Theorem \ref{thm}]
		Let $x\in\Omega$. Since $\Omega$ tiles with $\{0,\dots,a-1\}$, we have that $x+a\in \Omega+j$ for some $j\in\{0,\dots,a-1\}$. If $j\neq 0$, then 
		$x+(a-j)\in (\Omega+a-j)\cap \Omega$ , and since $a-j\in\{1,\dots,a-1\}$, this contradicts the fact that $\Omega$ tiles. Thus $x+a\in\Omega$, and $\Omega+a\subseteq\Omega$. 
		
		Similarly, if $x-a\in \Omega+j$ with $1\leq j\leq a-1$, then, with the previous step, $x\in\Omega+a+j\subseteq \Omega+j$, so $\Omega\cap(\Omega+j)\neq \ty$, again, in contradiction with the fact that $\Omega$ tiles. Thus $\Omega-a\subseteq\Omega$, so $\Omega\subseteq \Omega+a$, and this means that $\Omega=\Omega+a$ and $\Omega$ is invariant for $a\bz$. 
		
		Now take $\Omega_0:=\Omega\cap (0,a)$. Then $\Omega=\Omega_0+a\bz$ and the sets $\Omega_0+ak$, $k\in\bz$ are mutually disjoint. Thus, with Theorem \ref{th2}, we get that $\Omega$ is spectral with pair measure  $a\Leb_{[-\frac1{2a},\frac1{2a}]+\bz}$. 
	\end{proof}
	
	\begin{remark}\label{rem1}
		In general, if $A$ is a $d\times d$ integer matrix, $\mathcal R$ is a complete set of representatives $\mod A\bz^d$, and $\Omega$ is some open set in $\br^d$ which tiles by $\mathcal R$, it does not follow that $\Omega$ is also invariant for the lattice $A\bz^d$; this can be seen in Examples \ref{ex1} and \ref{ex2}.
		In the next paragraphs, we present some cases in which, if a set $\Omega$ tiles with a finite set, then $\Omega$ is invariant/periodic. 
	\end{remark}

	\section{Periodicity}\label{sec4}
	According to \cite{New77}, if a finite subset $\bz$ tiles $\bz$ by translations, then it is periodic. We will use this well known result to prove some invariance properties of subsets of $\br$ which tile with a finite set. 
	
	\begin{theorem}\label{thper}
		If a measurable subset $\Omega$ of $\br$ tiles $\br$ with a finite subset $B$ of $\bz$, then  $\Omega$ is periodic, i.e., $\Omega+p=\Omega$, for some positive integer $p$.
	\end{theorem}
	\begin{proof}
		
			Fix $t\in\br$. Denote 
		$$\mathcal T_t=\{n\in\bz : t+n\in\Omega\}.$$
		We claim that, for a.e. $t\in\br$, $B\oplus\mathcal T_t=\bz$. 
		
			Let $n\in\bz$. Excluding a set of measure zero of values of $t$, there is a unique $b\in B$, such that $t+n\in\Omega+b$. This implies that $t+(n-b)\in\Omega$, which means that $n-b\in\mathcal T_t$ and $n=b+(n-b)\in B+\mathcal T_t$. Thus $B+\mathcal T_t=\bz$. 
		
		If $n=b'+k'$ with $b'\in B$ and $k'\in\mathcal T_t$ then $t+k'\in\Omega$ and therefore $t+n=t+k'+b'\in\Omega+b'$. But we already have that $t+n\in\Omega+b$, and therefore $b=b'$ hence also $k'=n-b'=n-b$. This shows that $B\oplus\mathcal T_t=\bz$.

		By \cite{New77}, $\mathcal T_t$ has a period $P_{t}$, bounded by $2^{\max B-\min B}=:M$ (see also \cite{CM99}). Denote by  $p_1,p_2,\dots, p_r$ {\it all} the numbers that appear as period $P_t$ for any one of the sets $\mathcal T_t$, $t\in\br$. Since these numbers are all less than $M$, we have only finitely many possibilities. Take $p=p_1p_2\dots p_r$; then $p$ is a period for {\it all } sets $\mathcal T_t$, $t\in\br$, i.e., $\mathcal T_t+p=\mathcal T_t$. 
		
		We show that $\Omega$ has period $p$. Take $t\in\Omega$, then $0\in\mathcal T_t$, therefore $p\bz$ is contained in $\mathcal T_t$. This means that $t+pk\in\Omega$ for all $k\in\bz$, and so $\Omega$ is periodic.

	\end{proof}

		\begin{example}
		\label{exco2}
		Theorem \ref{thper} is not true in dimension 2. Here is an example from \cite[page 11]{KoMa10} of a set $\Omega$ which tiles $\br^2$ with a finite set in $\bz^2$, but is not periodic.
		
		First, we define a subset $\Gamma$ of $\bz^2$ which tiles $\bz^2$ with the finite set 
		$$\mathcal R:=\{(0,0),(2,0),(0,2),(2,2)\}.$$
		
		Define $\mathcal T_0=(4\bz)^2$. Clearly $\mathcal T_0\oplus \mathcal R=(2\bz)^2$. 
		
		Now we define two other sets  $\mathcal T_{\uparrow}$ and  $\mathcal T_{\rightarrow}$ which tile $(2\bz)^2$ by $\mathcal R$, but such that $\mathcal T_{\uparrow}$ does not have horizontal periods, and $\mathcal T_{\rightarrow}$ does not have vertical periods. 
		
		For $\mathcal T_{\uparrow}$, just take $\mathcal T_0$ and shift only the column at $0$ up by $2$; so 
		$$\mathcal T_{\uparrow}=(4\bz)^2\setminus\{(0,4k) : k\in\bz\}\cup\{(0,4k+2) : k\in\bz\}.$$
		For $\mathcal T_{\rightarrow}$, take $\mathcal T_0$ and shift only the row at $0$ right by $2$; so
		$$\mathcal T_{\rightarrow}=\mathcal (4\bz)^2\setminus\{(4k,0) : k\in\bz\}\cup\{(4k+2,0): k\in\bz\}.$$
		
		So $\mathcal T_{\uparrow}$ ruins horizontal periods, and $\mathcal T_{\rightarrow}$ ruins vertical periods. All three sets $\mathcal T_0$, $\mathcal T_{\uparrow}$ and $\mathcal T_{\rightarrow}$ tile $(2\bz)^2$ by $\mathcal R$. 
		
		Now, $(2\bz)^2$ tiles $\bz^2$ with $B=\{(0,0),(1,0),(0,1),(1,1)\}$. In each of these four disjoints cosets of $(2\bz)^2$ we will pick one of the $\mathcal T$-sets, to ruin the possible periodicity, as follows:\\
		In $(0,0)+(2\bz)^2$ we take $(0,0)+\mathcal T_0$;\\
		In $(1,0)+(2\bz)^2$ we take $(1,0)+\mathcal T_0$;\\
		In $(0,1)+(2\bz)^2$ we take $(0,1)+\mathcal T_{\uparrow}$;\\
		In $(1,1)+(2\bz)^2$ we take $(1,1)+\mathcal T_{\rightarrow}$.

		We define 
		$$\Gamma=\left((0,0)+\mathcal T_0\right)\cup \left((1,0)+\mathcal T_0\right)\cup \left((0,1)+\mathcal T_{\uparrow}\right)\cup\left((1,1)+\mathcal T_{\rightarrow}\right).$$
		
		We check that $ \Gamma$ tiles $\bz^2$ with $\mathcal R$. We have the disjoint union:
		$$\Gamma\oplus \mathcal R=\left((0,0)+\mathcal T_0\oplus\mathcal R\right)\cup \left((1,0)+\mathcal T_0\oplus\mathcal R\right)\cup \left((0,1)\mathcal T_{\uparrow}\oplus \mathcal R\right)\cup\left((1,1)+\mathcal T_{\rightarrow}\oplus\mathcal R\right)$$
		$$=\{(0,0),(1,0),(0,1),(1,1)\}\oplus(2\bz)^2=\bz^2.$$

		Now take 
		$$\Omega=(0,1)^2\oplus\Gamma.$$
		Clearly $\Omega$ tiles $\br^2$ by $\mathcal R$. We check that $\Omega$ is not periodic.

			\begin{figure}[ht]
			\centering
			\includegraphics[width=0.8\textwidth]{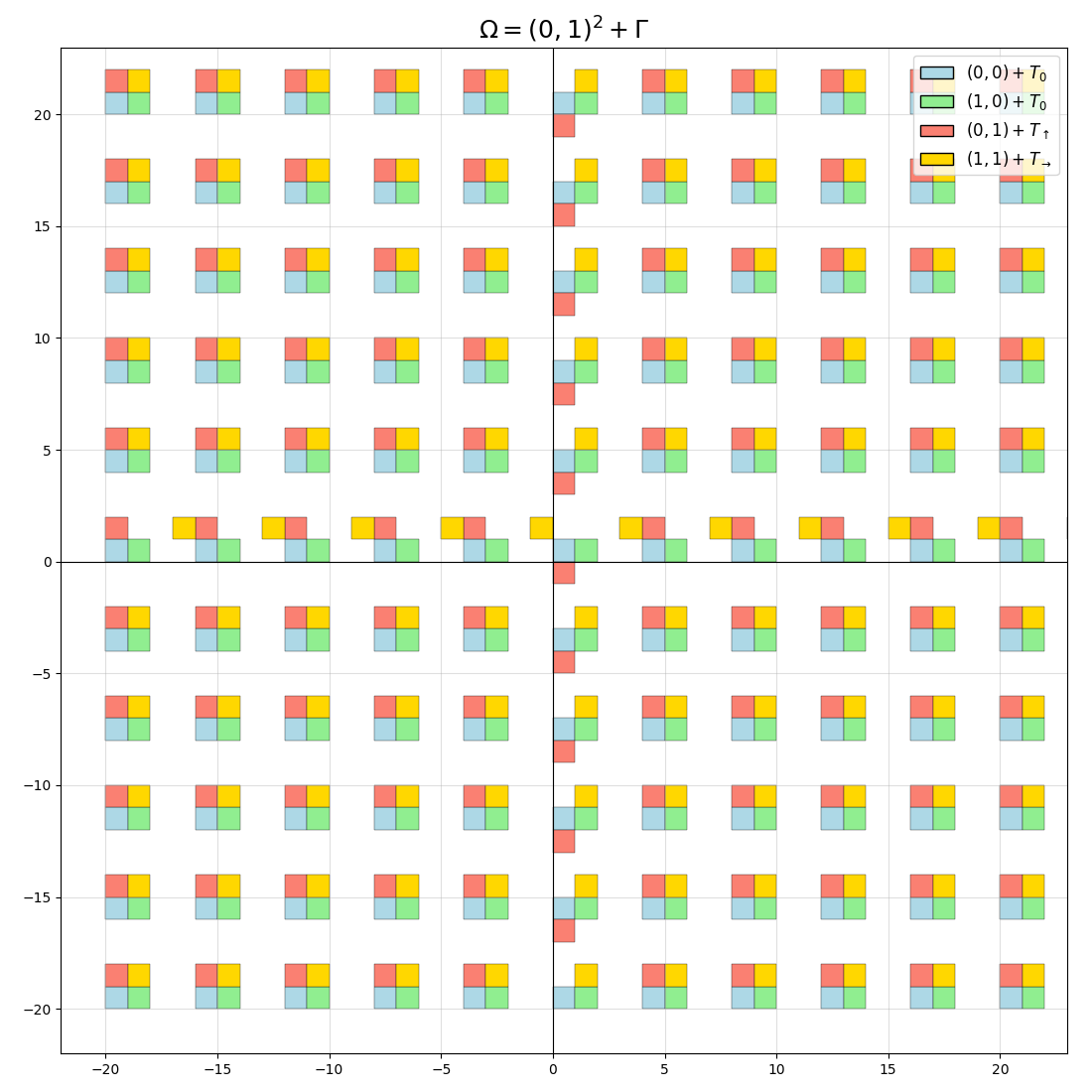}
			\caption{The nonperiodic set $\Omega=(0,1)^2+\Gamma$}
			\label{fig:nonper}
		\end{figure}
		
		Assume by contradiction that $(p_1,p_2)\neq 0$ is a period for $\Omega$. By taking points in $\Omega$ close to the boundary of the squares we can see that $p_1$ and $p_2$ have to be integers. Then $(p_1,p_2)$ is a period for $\Gamma$. 
		
		\newcommand{\circled}[1]{\textcircled{\small #1}}
		
		Note that the points in $\Gamma$ can be put into 4 disjoint buckets:\\
		$\circled 1$: $(4k_1,4k_2)$, $k_1,k_2\in\bz$;\\
		$\circled 2$: $(4k_1+1,4k_2)$ , $k_1,k_2\in\bz$;\\
		$\circled 3$: $(4k_1,4k_2+1)$, $k_1,k_2\in\bz$, $k_1\neq 0$ or $(0,4k_2+3)$, $k_2\in\bz$;\\
		$\circled 4$: $(4k_1+1,4k_2+1)$, $k_1,k_2\in\bz$, $k_2\neq 0$ or $(4k_1+3,1)$, $k_1\in\bz$.

		Note that we cannot have $p_2\equiv 0\mod 4$, because then $(0,3)\in\circled 3$, but $(0,3)+(p_1,p_2)=(p_1,p_2+3)\in\Gamma$ has $(p_2+3)\equiv 3\mod 4$, which forces $p_1=0$. Also $(3,1)\in\circled 4$ so $(3,1)+(0,p_2)=(3,p_2+1)\in\Gamma$, and this forces $p_2=0$.  
		
		We cannot have $p_2\equiv 1\mod 4$, because then $(4,1)\in\circled 3$, so $(4+p_1,1+p_2)\in \Gamma$ has $1+p_2\equiv 2\mod 4$ and so it cannot be in $\Gamma$. 
		
		We cannot have $p_2\equiv 2\mod 4$, because then $(0,0)\in\circled 1$ and so $(p_1,p_2)\in\Gamma$, but since $p_2\equiv 2\mod 4$, this cannot be in $\Gamma$. 
		
		We cannot have $p_2\equiv 3\mod 4$, because then $(0,3)\in\circled 3$ and so $(p_1,3+p_2)\in\Gamma$, but this has $3+p_2\equiv 2\mod 4$, and thus it cannot be in $\Gamma$.
		
		In conclusion $\Gamma$ is not periodic and so neither is $\Omega$.
	\end{example}
	
	Looking at the proof of Theorem \ref{thm}, we see that if $\Omega$ tiles $\br$ by the set $\{0,1,\dots,a-1\}$ for some integer $a\geq 2$, then $\Omega$ has period 2. In general, it is not true that if a set $\Omega$ tiles $\br$ with some complete set of representatives modulo $a$, then $\Omega$ has period $a$, see Examples \ref{ex1} and \ref{ex2}. However if $a=p$ is prime, and the set of complete representatives has greatest common divisor 1, then $\Omega$ will have period $p$. This is stated in the next theorem. 
	
	\begin{theorem}
		\label{thinv}
		Suppose $\Omega$ is a Lebesgue measurable set which tiles $\br$ with a subset $B$ of $\bz$ which has the following properties: $|B|=p$ prime, $0\in B$ and $\gcd(B)=1$. Then $B$ is a complete set of representatives modulo $p$ and $\Omega$ has period $p$.
	\end{theorem}

	\begin{proof}
		Fix $t\in\br$. Denote 
		$$\mathcal T_t=\{n\in\bz : t+n\in\Omega\}.$$
		As we saw in the proof of Theorem \ref{thper}, we have that $B\oplus \mathcal T_t=\bz$. 
		
		We need a lemma 	{(see also \cite[Theorem 2]{Tij95})}:

		\begin{lemma}
			\label{lem16}
			
			Suppose $B\oplus\mathcal T=\bz$, $|B|=p$ prime, $0\in B\cap\mathcal T$, $\gcd(B)=1$. Then $\mathcal T=p\bz$ and $B$ is a complete set of representatives modulo $p$. 
		\end{lemma}
		
		\begin{proof}
			We know from \cite{New77} that $\mathcal T$ has to be periodic and let $N$ be its minimal period. Then $\mathcal T=C\oplus N\bz$ and $B\oplus C=\bz_N$, $0\in C$, for some finite set $C\subseteq\{0,1,\dots,N-1\}$. 
			
			We use the following theorem:

			\begin{theorem}\cite[page 69]{San57}
				\label{thsan}
				If $B\oplus C=\bz_N$ and $|B|=p^\lambda$ for some prime number $p$ and some $\lambda\in\bn$, then either $B$ or $C$ is periodic. 
			\end{theorem}
			Since $|B|=p$ is prime, by Theorem \ref{thsan}, it follows that either $B$ or $C$ is periodic in $\bz_N$. However $C$ cannot be periodic, because that would contradict the fact that $N$ is the minimal period of $\mathcal T$. Thus $B$ has to be periodic. Let $H$ be the subgroup of periods of $B$, in $\bz_N$. Then $B$ is a union of cosets of $H$ and therefore the order of $H$ has to divide the cardinality of $B$, which is $p$. Thus $|H|=|B|=p$ and, since $0\in B$, $B=H$. But, the only subgroup of order $p$ of $\bz_N$ is $\{0,N/p,2N/p,\dots,(p-1)N/p\}$, and so $B= \{0,N/p,2N/p,\dots,(p-1)N/p\}$ modulo $N$. But then, all elements in $B$ are divisible by $N/p$, and since $\gcd(B)=1$, we must have $N=p$. Finally, this implies that $|C|=1$ and so $C=\{0\}$, therefore $\mathcal T=p\bz$. Since $B\oplus p\bz=\bz$, it follows that $B$ is a complete set of representatives modulo $p$.
		\end{proof}

		Consider now $t\in\Omega$, so $0\in\mathcal T_t$. With Lemma \ref{lem16}, we obtain that $\mathcal T_t=p\bz$. This means $t+pk\in\Omega$ for all $k\in\bz$. This shows that $\Omega+p\bz\subseteq\Omega$ and the conclusion follows. 
	\end{proof}

	\section{Examples}\label{sec5}

	\begin{example}\label{ex1}
		Here is an example of a set $\Omega$ which tiles by the complete set of representatives modulo 2, $\mathcal R=\{0,3\}$, but is not invariant under addition by 2, $\Omega+2\neq \Omega$, 
		$$\Omega=(0,3)+6\bz.$$
		
	 We can apply Theorem \ref{th1} with $\Omega_0=(0,3)$ with spectrum $\Lambda_0=\frac13\bz$ and $A=6$. We obtain that $\Omega$ is spectral with pair measure $2\Leb_{[-\frac1{12},\frac1{12}]+\frac13\bz}$.

		Note that this set $\Omega=(0,3)+6\bz$ does not have as pair measure the renormalized Lebesgue measure on $[-1/4,1/4]+\bz$, as $\Omega'=(0,1)+2\bz$ does (Theorem \ref{thm}), even though both these sets tile $\br$ by $\{0,3\}$.
		
		Indeed, if we take the characteristic function $\chi_{(0,1)}$ and its translation $\chi_{(1,2)}$, we have $\ip{\chi_{(0,1)}}{\chi_{(1,2)}}=0$. Assuming that the Fourier transform is isometric from $L^2(\Omega)$ into $L^2(2\Leb_{[-\frac1{4},\frac1{4}]+\bz})$, we would get that 
		$\ip{\hat\chi_{(0,1)}}{e^{-2\pi it}\hat\chi_{(0,1)}}=0$. 

		We have 
		$$\hat\chi_{(0,1)}(t)=e^{-\pi it}\sinc(\pi t)=e^{-\pi i t}\frac{\sin \pi t}{\pi t},\quad(t\in\br).$$
		Since $\bz$ is a spectrum for $(0,1)$, with Lemma \ref{lemma1} we have that 
				$$1=\sum_{k\in\bz}|\hat\chi_{(0,1)}(t+k)|^2=\sum_{k\in\bz}\sinc^2(\pi (t+k)),\quad(t\in\br).$$
	
	Then 
	$$0=\ip{\hat\chi_{(0,1)}}{e^{-2\pi it}\hat\chi_{[0,1]}}=\int_{[-1/4,1/4]+\bz}e^{2\pi it}\sinc^2(\pi t)\,dt=\sum_{k\in\bz}\int_{-1/4}^{1/4}e^{2\pi i (t+k)}\sinc^2(\pi (t+k))\,dt$$
	$$=\int_{-1/4}^{1/4}e^{2\pi it}\sum_{k\in\bz}\sinc^2(\pi (t+k))\,dt=\int_{-1/4}^{1/4}e^{2\pi it}\,dt=\frac{e^{2\pi i/4}-e^{-2\pi i/4}}{2\pi i}\neq 0,$$
		a contradiction. Thus $2\Leb_{[-\frac1{4},\frac1{4}]+\bz}$ is not a pair measure for $(0,3)+6\bz$.
		
	\end{example}

		We generalize the idea in Example \ref{ex1} into a lemma, which gives a necessary condition for a measure to be a pair measure for a set $\Omega$. 
	
	\begin{lemma}
		\label{lemnec}
		Let $\Omega$ be a spectral measurable subset of $\br^d$ with pair measure of the form $\nu_0\ast\Sha_{\Lambda_0}$ where $\nu_0$ is a finite positive Radon measure on $\br^d$ and $\Lambda_0$ is a countable subset of $\br^d$. 
		
		Assume in addition, that there is a subset $E$ of $\br^d$ of finite measure and $a\in\br^d$ with the following properties:
		
		\begin{enumerate}
			\item Both $E$ and $E+a$ are subsets of $\Omega$. 
			\item The set $E$ has spectrum $\Lambda_E$, and $\Lambda_E$ tiles $\Lambda_0$ with a finite set $F$. 
			\item $\lambda_E\cdot a\in\bz$ for all $\lambda_E\in\Lambda_E$.
		\end{enumerate}
		Then  
		$$\frac{|E\cap(E+a)|}{|E|^2}=\hat\nu_0(a)\sum_{f\in F}e^{-2\pi i a\cdot f}.$$
	\end{lemma}

	\begin{proof}
		We have 
		$$|E\cap(E+a)|=\ip{\chi_{E+a}}{\chi_E}_{L^2(\Omega)}=\ip{\hat\chi_{E+a}}{\hat\chi_{E}}_{L^2(\nu_0\ast\Sha_{\Lambda_0})}=\int e^{-2\pi ia\cdot x}|\hat\chi_E(x)|^2\,d\nu_0\ast\Sha_{\Lambda_0}(x)$$
		$$=\int\sum_{\lambda\in\Lambda_0}e^{-2\pi ia\cdot(t+\lambda)}|\hat\chi_E(t+\lambda)|^2\,d\nu_0(t)
		=\int\sum_{f\in F}\sum_{\lambda_E\in\Lambda_E}e^{-2\pi ia\cdot(t+f+\lambda_E)}|\hat\chi_E(t+f+\lambda_E)|^2\,d\nu_0(t)$$
		$$=\int\sum_{f\in F}e^{-2\pi ia\cdot(t+f)}\sum_{\lambda_E\in\Lambda_E}|\hat\chi_E(t+f+\Lambda_E)|^2\,d\nu_0(t)$$
		but, with Lemma \ref{lemma1}, $\sum_{\lambda_E}|\chi_E(t+f+\lambda_E)|^2=|E|^2$, so we get further
		$$=|E|^2\int\sum_{f\in F}e^{-2\pi ia\cdot (t+f)}\,d\nu_0(t)=|E|^2\sum_{f\in F}e^{-2\pi i a\cdot f} \int e^{-2\pi ia\cdot t}\,d\nu_0(f)$$
		$$=|E|^2\hat\nu_0(a)\sum_{f\in F}e^{-2\pi ia\cdot f}.$$
		The conclusion follows.
		
	\end{proof}

	\begin{example}\label{ex2}
		Here is an example of a complete set of representatives modulo 4: 
		$$\mathcal R=\{0,5,6,11\},$$
		with greatest common divisor equal to 1, and a set $\Omega$ which tiles with $\mathcal R$ but it is not invariant under addition by 4, $\Omega+4\neq \Omega$, where $\Omega$ is defined by:
		$$B_0:=\{0,2,4,12,14,16\},\quad B=B_0+24\bz,\quad \Omega=B+(0,1).$$
		
		We can apply Theorem \ref{th1} with $\Omega_0=(0,1)+B_0$, $A=24$. With Lemma \ref{lemco}, since $B_0=2\{0,1,2\}+12\{0,1\}$, and since $\Sha_{2\{0,1,2\}}$ has spectrum $\frac16\{0,1,2\}$ and $\Sha_{12\{0,1\}}$ has spectrum $\frac1{24}\{0,1\}$, we get that $\Sha_{B_0}$ has spectrum $\frac16\{0,1,2\}+\frac1{24}\{0,1\}=\frac1{24}\{0,1,4,5,8,9\}$. With Lemma \ref{lemco} again, we obtain that $\Omega_0=(0,1)+B_0$ has spectrum $\Lambda_0=\bz+\frac1{24}\{0,1,4,5,8,9\}$.
		
		Then $\Omega=\Omega_0+24\bz$ is spectral with pair measure $4\Leb_{[-\frac1{48},\frac1{48}]+\bz+\frac1{24}\{0,1,4,5,8,9\}}$.

		Note that $B_0\oplus \mathcal R=\bz_{24}$. And from this, we get that $B\oplus \mathcal R=B_0\oplus\mathcal R\oplus 24\bz=\bz$ and then $\Omega\oplus \mathcal R=B+(0,1)+\mathcal R=\bz+(0,1)=\br.$ Thus $\Omega$ tiles $\br$ with $\mathcal R$. 
		
		If we take $\Omega'=(0,1)+4\bz$, then $\Omega'$ tiles $\br$ with any complete set of representatives mod 4, in particular with $\mathcal R$ and with $\{0,1,2,3\}$. Then, with Theorem \ref{thm}, $\Omega'$ has pair measure $4\Leb_{[-1/8,1/8]+\bz}$. We show that this is {\it not} a pair measure for $\Omega$. Assuming, by contradiction that this is the case, we can use Lemma \ref{lemnec} with $\nu_0=4\Leb_{[-1/8,1/8]}$, $\Lambda_0=\bz$, $E=(0,1)$ with spectrum $\Lambda_E=\bz$, and the finite set $F=\{0\}$ and $a=2$. 
		
		Since $E\cap (E+2)=\ty$, we should have either $0=\sum_{f\in F}e^{-2\pi ia\cdot f}$  or $\hat\nu_0(2)=0$. We do not have the former since $\sum_{f\in F}e^{-2\pi ia\cdot f}=e^{-2\pi i2\cdot 0}=1$. We do not have the latter since  
		$$\int_{-1/8}^{1/8}e^{-2\pi i2t}\,dt=\frac{e^{-2\pi i\frac18}-e^{+2\pi \frac18}}{-4\pi i}\neq 0.$$

	\end{example}

	\begin{example}
		\label{exfu}
		Suppose now we want, in Theorem \ref{th1}, to take $\Omega_0$ to be a fundamental domain $F_B$ for some lattice $B\bz^d$, where $B$ is a real $d\times d$ invertible matrix and then translate by another lattice $A\bz^d$. We will need $F_B$ to pack with the lattice $A\bz^d$. Since $F_B$ tiles with $B\bz^d$, with Fuglede's Theorem \ref{thf1}, $F_B$ has spectrum $\Lambda_0=(B^T)^{-1}\bz^d$, the dual lattice to $B\bz^d$. We need to make sure that the condition \eqref{eq1.2} is also satisfied: $(B^T)^{-1}\bz^d\subseteq \ati\bz^d$. We have the following lemma:
			\begin{figure}[ht]
			\centering
			\includegraphics[width=0.8\textwidth]{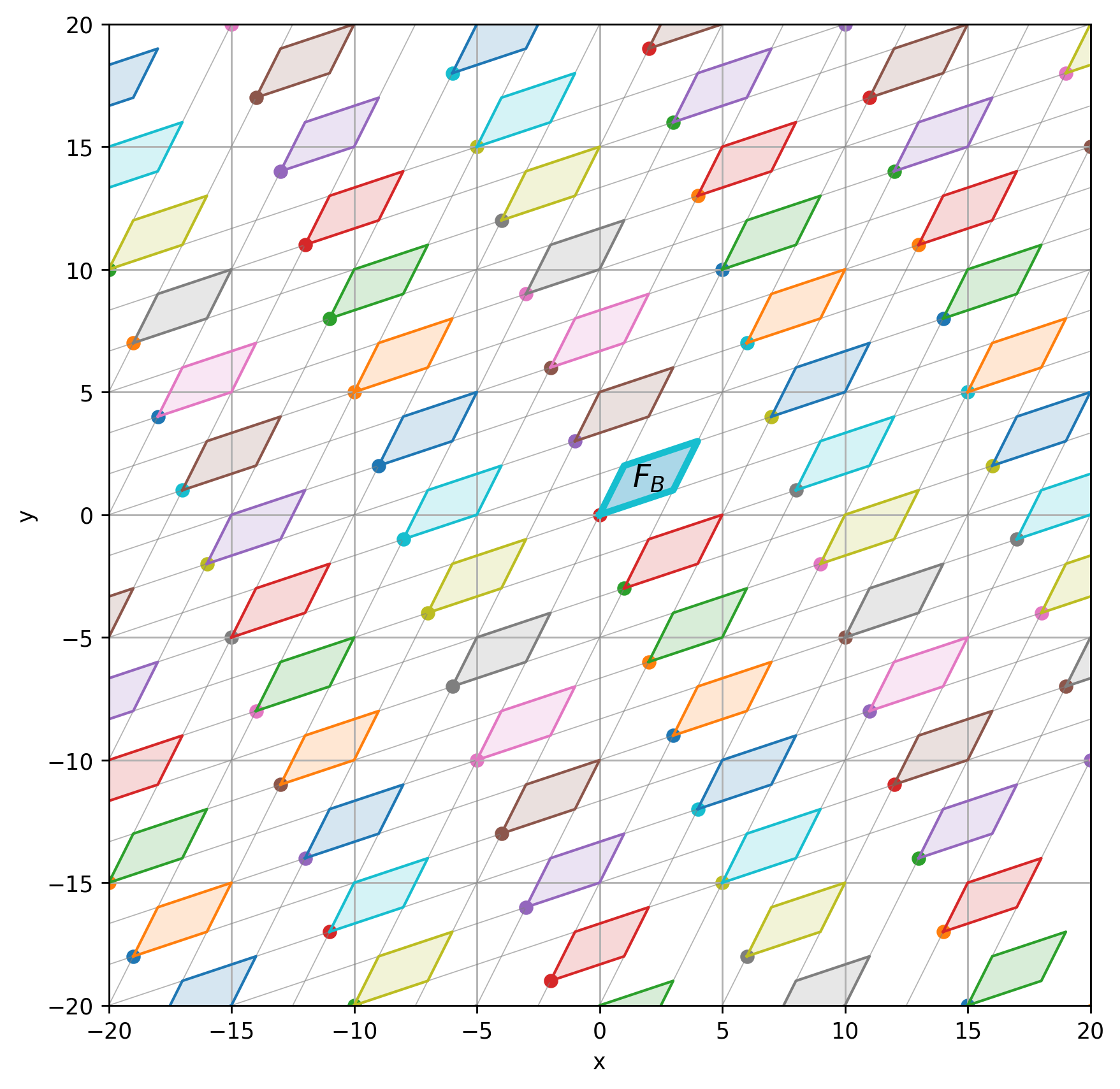}
			\caption{The set $F_B + A\mathbb{Z}^2$ together with the lattice $B\mathbb{Z}^2$ shown by the intersection points of the thin grey lines. The highlighted parallelogram is the fundamental domain $F_B$. The dots represent the lattice $A\bz^2$.}
			\label{fig:FB-AZ2-example5}
		\end{figure}

		\begin{figure}[ht]
			\centering
			\includegraphics[width=0.82\textwidth]{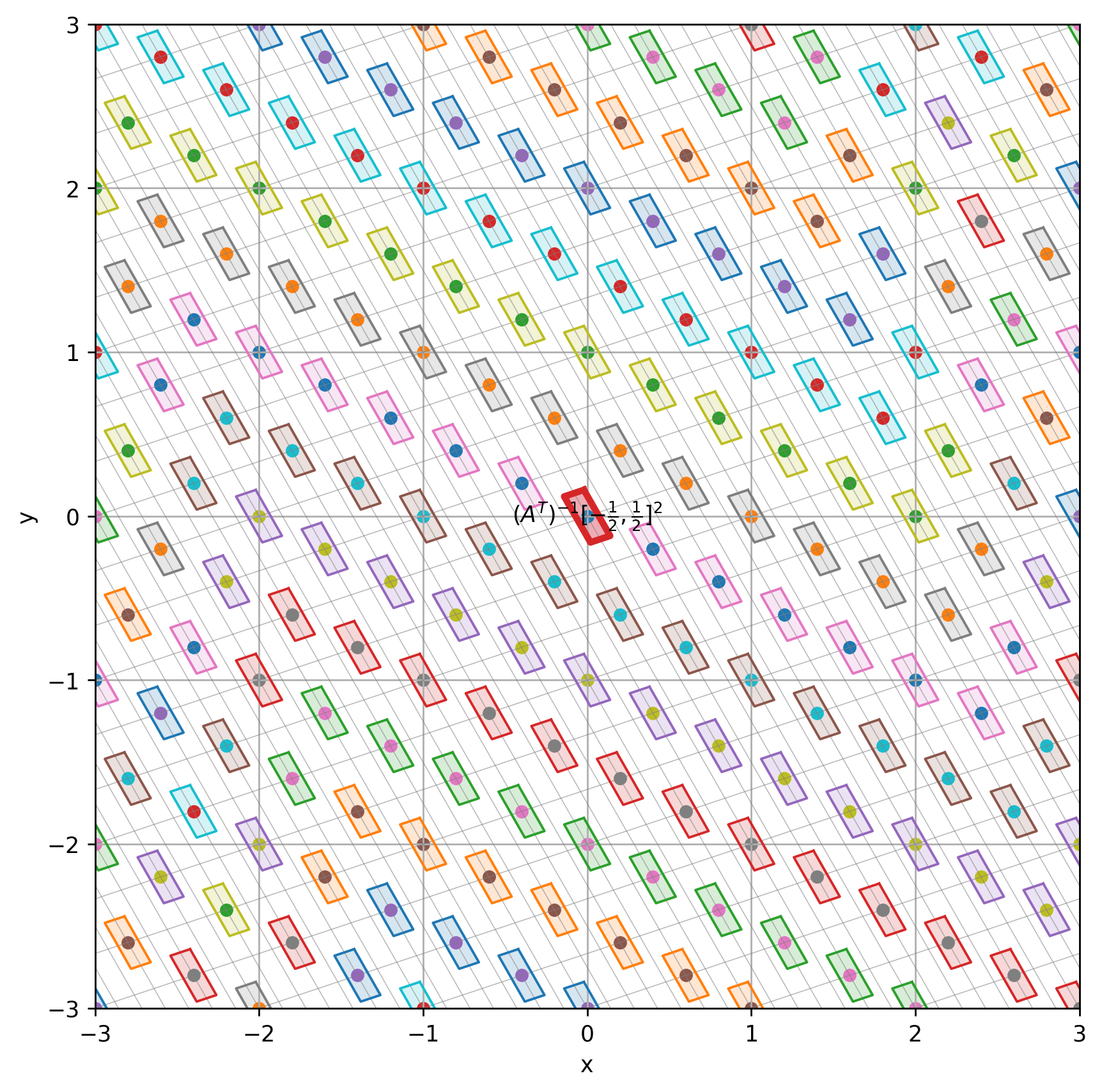}
			\caption{
				The translates $(A^T)^{-1}\Bigl[-\frac12,\frac12\Bigr]^2+(B^T)^{-1}\mathbb Z^2$. The thin grey lines represent the lattice $(A^T)^{-1}\mathbb Z^2$. The dots represent the lattice $(B^T)^{-1}\bz^2$. 
			}
			\label{fig:dual-lattice-tiling}
		\end{figure}
		
		\end{example}
		
		\begin{lemma}
			\label{lemexfu}
			The following statements are equivalent:
			\begin{enumerate}
				\item  $(B^T)^{-1}\bz^d\subseteq \ati\bz^d$.
				\item $B^{-1}A$ has integer entries. 
				\item $A=BM$ for some integer $d\times d$ matrix $M$. 
				\item The lattice $A\bz^d$ is contained in the lattice $B\bz^d$. 
			\end{enumerate}
			
			In this case a fundamental domain $F_B$ for the lattice $B\bz^d$ packs with the lattice $A\bz^d$. 
		\end{lemma}
		
		\begin{proof}
			We have  $(B^T)^{-1}\bz^d\subseteq \ati\bz^d$ iff $A^T(B^T)^{-1}\bz^d\subseteq \bz^d$ iff $A^T(B^T)^{-1}=(B^{-1}A)^T$ has integer entries iff $M=B^{-1}A$ has integer entries iff $B^{-1}A\bz^d\subseteq \bz^d$ iff $A\bz^d\subseteq B\bz^d$. 
			
			If these conditions are satisfied then, since $F_B$ tiles with $B\bz^d$ and $A\bz^d \subseteq B\bz^d$, it follows that $F_B$ packs with $A\bz^d$. 
		\end{proof}
		
		Note that the measure of $\Omega_0=F_B$ is $|\Omega_0|=|F_B|=|\det B|$, so 
		$$\frac{|\det A|}{|\Omega_0|}=\frac{|\det A|}{|\det B|}=|\det M|=[A\bz^d : B\bz^d]=\mbox{ the index of $B\bz^d$ in $A\bz^d$}.$$
		
		Thus, if these conditions are satisfied then, with Theorem \ref{th1}, the set $\Omega=F_B+A\bz^d$ is spectral with pair measure 
		
		$$[A\bz^d : B\bz^d]\Leb_{\ati[-1/2,1/2]^d+(B^T)^{-1}\bz^d}.$$

		For example, let

		$$
		B=
		\begin{pmatrix}
			3 & 1\\
			1 & 2
		\end{pmatrix},
		\qquad
		M=
		\begin{pmatrix}
			2 & -1\\
			1 & 2
		\end{pmatrix}		
		\qquad		
		A=BM=
		\begin{pmatrix}
			7 & -1\\
			4 & 3
		\end{pmatrix}
		$$. 
		
		The set $\Omega=F_B+A\bz^2$ is represented in Figure \ref{fig:FB-AZ2-example5}; the support of the pair measure is represented in Figure \ref{fig:dual-lattice-tiling}. The renormalization constant is $|\det M|=5$.

	\begin{example}
		\label{exsa}
		A particularly appealing class of examples of sets $\Omega_0$ for Theorem \ref{th1} or Theorem \ref{th2} can be obtained from self-affine tiles (see \cite{LW97} for the definitions and details for what follows in this example). An {\it integer self-affine tile} is associated to a expansive $d\times d$ matrix $R$ (expansive means all eigenvalues have absolute value strictly greater than 1), and a set of digits $\mathcal D$ in $\bz^d$ with $|\mathcal D|=|\det R|$. Then there is a unique compact set $T=T(R,\mathcal D)$ that satisfies the equality 
		$$R(T)=\bigcup_{d\in\mathcal D}(T+d),$$
		which is given explicitly by 
		$$T(R,\mathcal D):=\left\{\sum_{k=1}^\infty R^{-k}d_k : \mbox{ all }d_k\in\mathcal D\right\}.$$
		The set $T(R,\mathcal D)$ is called a {\it self-affine tile} if it has positive Lebesgue measure. 
		
		Associated to an integral pair $(R,\mathcal D)$, is the smallest $R$-invariant sublattice $\bz[R,\mathcal D]$ of $\bz^d$ that contains the difference set $\mathcal D-\mathcal D$. When $0\in\mathcal D$, this sublattice is 
		$$\bz[R,\mathcal D]=\bz[\mathcal D,R(\mathcal D),\dots, R^{d-1}\mathcal D].$$
		A pair $(R,\mathcal D)$ is called {\it primitive} if $\bz[R,\mathcal D]=\bz^d$, and $\mathcal D$ is called a {\it primitive digit set} for $R$. 
		
		The study can be reduced to primitive digit sets: there is another integral self-affine tile $\tilde T=T(\tilde R,\tilde D)$ with $(\tilde R,\tilde D)$ primitive and $0\in\tilde D$, such that for some integer $d\times d$ matrix $B$ (whose columns form a basis for $\bz[R,\mathcal D]$) with $\det B\neq 0$, and some $v\in\bz^d$, 
		$$T=B(\tilde T)+v.$$
		
		A primitive digit set is called {\it standard} if it forms a complete set of residues ($\mod R$). 
		
		Theorem 1.1 in \cite{LW97} states that: 
		\begin{theorem}
			\label{thlw}
			Every integral self-affine tile $T$ coming from a standard digit set tiles $\br^d$ by some lattice $\Gamma\subseteq \bz^d$. 
		\end{theorem}

		We can construct our examples for Theorem \ref{th1} as follows: let $\Omega_0:=T(R,\mathcal D)=T$ for a standard digit set. Let $\tilde\Gamma$ be any full-rank sublattice of $\Gamma$; let $A$ be a matrix such that $A\bz^d=\tilde \Gamma$. Since $\Omega_0$ tiles with the lattice $\Gamma$, it packs with the lattice $\tilde\Gamma$. Also, by Fuglede's Theorem \ref{thf1}, $\Omega_0$ has spectrum $\Lambda_0=\Gamma^*$, the dual lattice of $\Gamma$, 
		$$\Gamma^*=\left\{\gamma^*\in\br^d :\gamma^*\cdot \gamma\in\bz\mbox{ for all }\gamma\in\Gamma	\right\}=\ati\bz^d.$$
		
		Since $\tilde\Gamma\subseteq\Gamma$ we also have $\lambda_0\cdot \tilde\gamma\in\bz$ for all $\lambda_0\in\Lambda_0=\Gamma^*$, and $\tilde \gamma\in\tilde\Gamma$.
		
		Note also that $|\det A|$ is the index of $\tilde \Gamma$ in $\bz^d$, $[\bz^d:\tilde\Gamma]$. As, $\Omega_0$ tiles $\br^d$ with $\Gamma$, using Lemma \ref{lemcg}, we get that $|\Omega_0|=[\bz^d: \Gamma]$. Thus 
		$$\frac{|\det A|}{|\Omega_0|}=\frac{[\bz^d:\tilde\Gamma]}{[\bz^d:\Gamma]}=[\Gamma:\tilde\Gamma].$$
		
		Thus, we can apply Theorem \ref{th1} and obtain that the set $\Omega=\Omega_0+\tilde\Gamma$ is spectral with pair measure 
		$$[\Gamma:\tilde\Gamma]\Leb_{\ati[-\frac12,\frac12]^d+\Gamma^*}$$
		
		Note that, if we take $\mathcal R$ to be a complete set of representatives of $\tilde\Gamma$ in $\Gamma$, then 
		$$\Omega\oplus\mathcal R=\Omega_0\oplus\tilde\Gamma\oplus\mathcal R=\Omega_0+\Gamma=\br^d,$$
		so $\Omega$ tiles $\br^d$ by $\mathcal R$. In particular, if $\Gamma=\bz^d$, then Theorem \ref{th2} can be used. 
		
		In the next figures, the first picture is the self-affine tile $\Omega_0=T(R,\mathcal D)$, the second illustrates how $\Omega_0$ tiles $\br^d$ with the lattice $\Gamma$, and the third picture shows (part of) the set $\Omega=\Omega_0+\tilde\Gamma$. Note that the tiling lattice $\Gamma$ can be different than $\bz^2$. 
		
			\begin{figure}[h]
			\centering
			
			\begin{minipage}{0.28\textwidth}
				\centering
				\includegraphics[width=\linewidth]{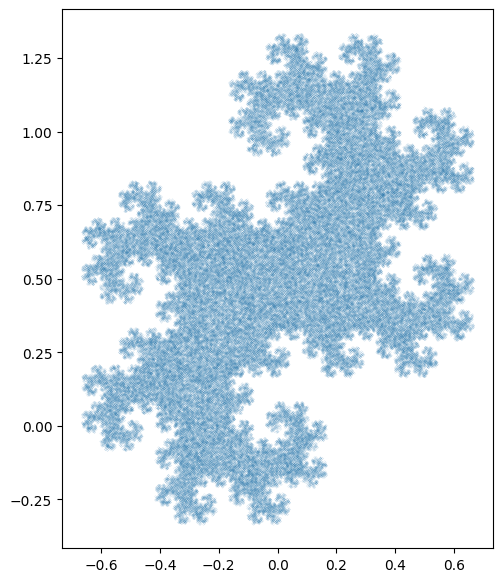}
			\end{minipage}
			\centering
			\begin{minipage}{0.28\textwidth}
				\centering
				\includegraphics[width=\linewidth]{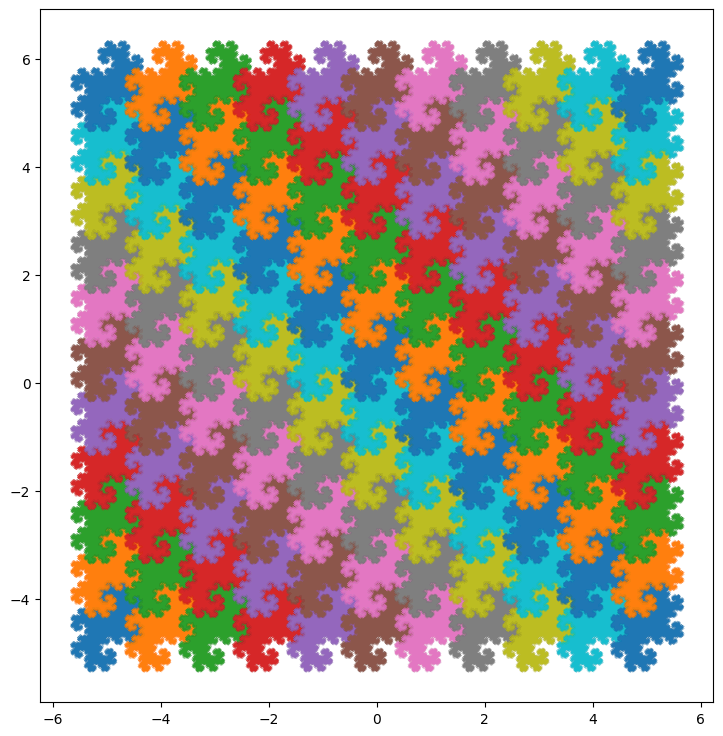}
			\end{minipage}
			\begin{minipage}{0.28\textwidth}
				\centering
				\includegraphics[width=\linewidth]{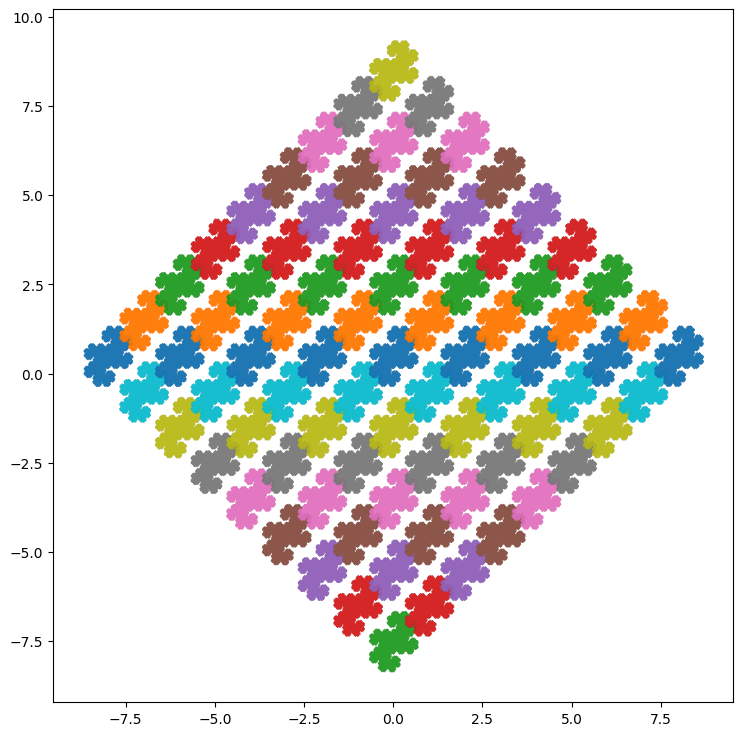}
			\end{minipage}
			
			\caption{\tiny{$
				R=
				\begin{pmatrix}
					1 & 1\\
					-1 & 1
				\end{pmatrix},
			\mathcal D=
				\{
				\binom{0}{0},
				\binom{1}{0}
				\}, \Gamma=\bz^2,\tilde \Gamma=R(\bz^2)$}}
			\label{fig:twin-dragon}
		\end{figure}

		\begin{figure}[h]
			\centering
			
			\begin{minipage}{0.28\textwidth}
				\centering
				\includegraphics[width=\linewidth]{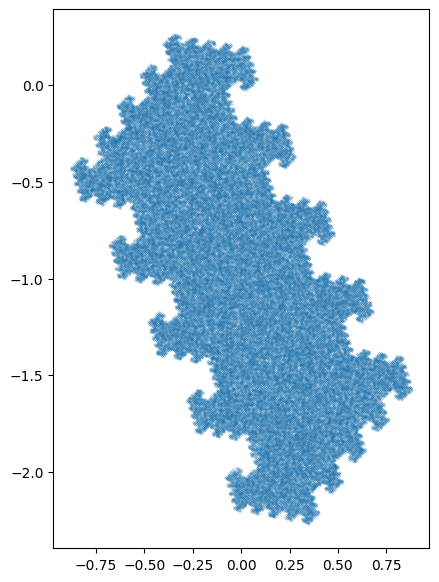}
			\end{minipage}
\centering
\begin{minipage}{0.28\textwidth}
	\centering
	\includegraphics[width=\linewidth]{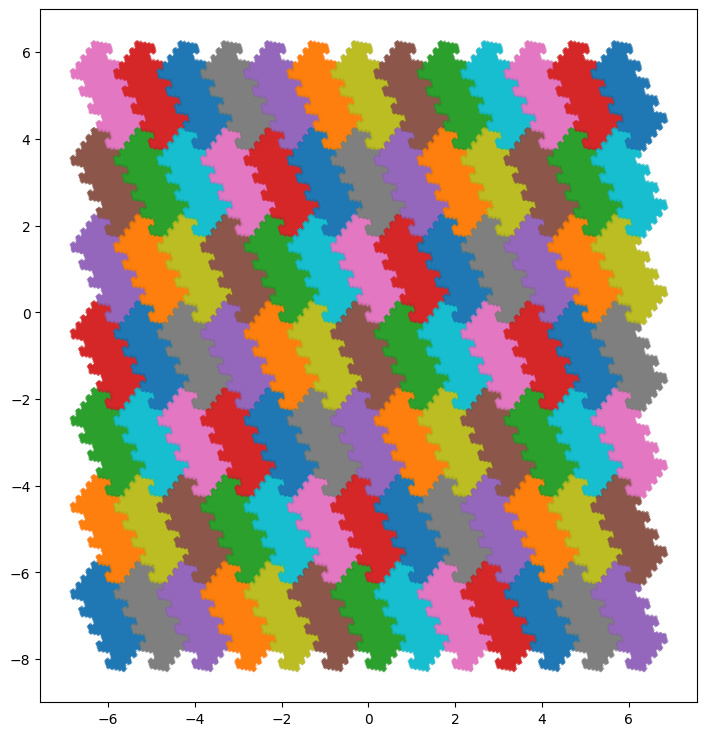}
\end{minipage}
			\begin{minipage}{0.28\textwidth}
				\centering
				\includegraphics[width=\linewidth]{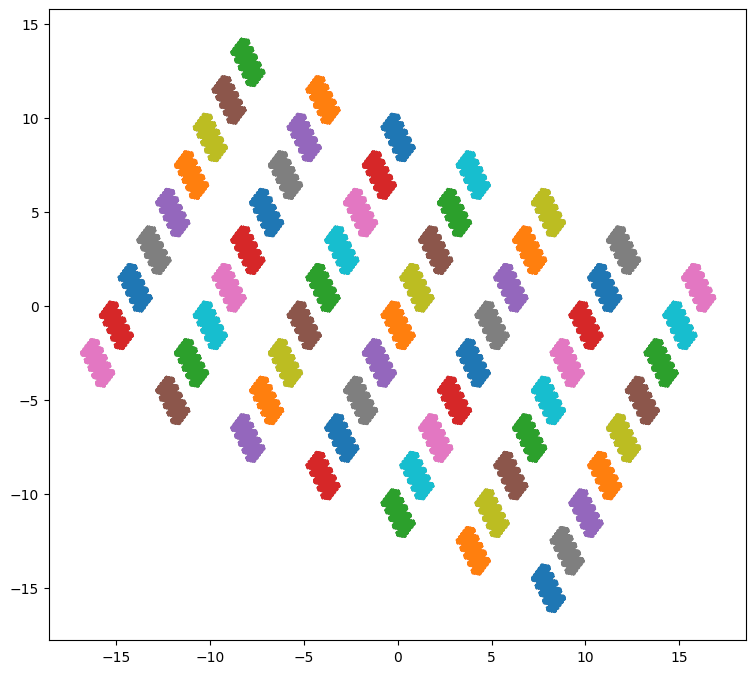}
			\end{minipage}
			
			\caption{\tiny{$
					R=
					\begin{pmatrix}
						1 & -2\\
						2 & 1
					\end{pmatrix},			
					\mathcal D=
					\{
					\binom{0}{0},
					\binom{1}{0},
					\binom{2}{0},
					\binom{3}{0},
					\binom{4}{0}
					\}, \Gamma=\bz\times 2\bz,\tilde \Gamma=R(\bz\times 2\bz)$}}
			\label{fig:measure3}
		\end{figure}

				\begin{figure}[h]
			\centering
			
			\begin{minipage}{0.28\textwidth}
				\centering
				\includegraphics[width=\linewidth]{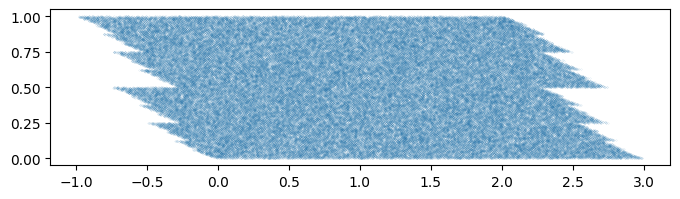}
			\end{minipage}
			\centering
			\begin{minipage}{0.28\textwidth}
				\centering
				\includegraphics[width=\linewidth]{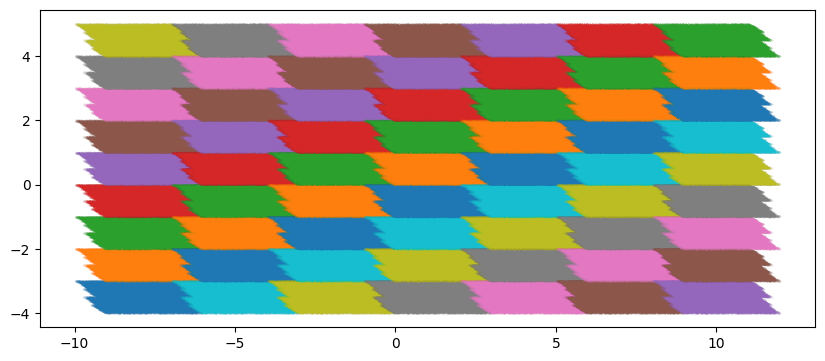}
			\end{minipage}
			\begin{minipage}{0.28\textwidth}
				\centering
				\includegraphics[width=\linewidth]{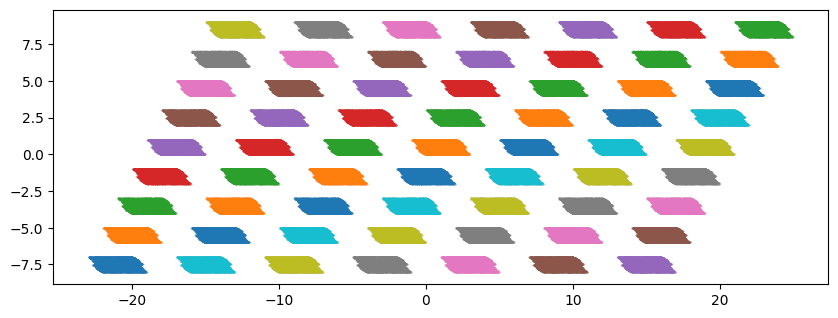}
			\end{minipage}
			
			\caption{\tiny{$
					R=
					\begin{pmatrix}
						2 & 1\\
						0 & 2
					\end{pmatrix},
					\qquad
					D=
					\{
					\binom{0}{0},
					\binom{3}{0},
					\binom{0}{1},
					\binom{3}{1}
					\}, \Gamma=3\bz\times \bz,\tilde \Gamma=R(3\bz\times \bz)$}}
			\label{fig:measure3}
		\end{figure}
		
	\end{example}
	
\FloatBarrier 
	
	
	
	
	
	
		
	\bibliographystyle{alpha}	
	\bibliography{eframes}
	
	\end{document}